\documentclass{amsart}
\usepackage{amsmath,amssymb,amsbsy}
\usepackage{cmll, fouridx}

\theoremstyle{plain}
\newtheorem{thm}{Theorem}[section] 
\newtheorem{lemma}[thm]{Lemma}     
\newtheorem{cor}[thm]{Corollary}
\newtheorem{prop}[thm]{Proposition}
\newtheorem{MainThm}{Theorem}

\theoremstyle{definition}
\newtheorem{defn}[thm]{Definition}
\newtheorem*{defn*}{Definition}

\newtheorem{fact}[thm]{Fact}

\newtheorem{claim}{Claim}[thm]

\newcommand{\Thfin}[1]{\mathrm{Th}_{\mathrm{fin}}(#1)}
\newcommand{\var}[1]{\mathcal{#1}}
\newcommand{\Th}[1]{\mathrm{Th}(#1)}
\newcommand{\alg}[1]{\boldsymbol{#1}}
\newcommand{\Con}[1]{\mathrm{Con}(#1)}
\newcommand{\Pol}[2]{\mathsf{Pol}_{#1}(#2)}
\newcommand{\HSP}[1]{\mathrm{HSP}(#1)}
\newcommand{\card}[1]{| #1 |}
\newcommand{\EQtwo}{\mathcal{E}_2}
\newcommand{\TC}[3]{\mathrm{C}(#1,#2;#3)}
\newcommand{\prectype}[1]{\stackrel{#1}{\prec}}
\newcommand{\typset}[1]{\mathrm{typ} \left\{ #1 \right\} }
\newcommand{\SSRad}[1]{\mathrm{Rad}_u(#1)}
\newcommand{\Rad}[1]{\mathrm{Rad}(#1)}
\newcommand{\aI}[1]{\mathbf{#1}}
\newcommand{\typ}[2]{\mathrm{typ}(#1,#2)}
\newcommand{\Cg}[3]{\mathrm{Cg}_{#1}(\langle #2,#3 \rangle)}
\newcommand{\Edge}{\stackrel{E}{\text{---}}}
\newcommand{\PolGrp}[2]{\mathfrak{S}_{#2}^{#1}}
\newcommand{\TwinGrp}[2]{\mathrm{T}_{#2}^{#1}}
\newcommand{\Eset}[3]{\mathrm{E}^{#1}(#2,#3)}
\newcommand{\HS}[1]{\mathrm{HS}(#1)}

\newcommand{\EQnum}{\arabic{section}.\arabic{thm}.\arabic{claim}}
\renewcommand{\bibname}{\textsc}
\renewcommand{\and}{\textup{and }}

\title[Residual finiteness of finitely decidable varieties]%
{Strong solvability and residual finiteness for finitely decidable varieties}

\author{Ralph McKenzie and Matthew Smedberg}
\address{Department of Mathematics, Vanderbilt University}
\email{rn.mckenzie@vanderbilt.edu}
\email{matthew.smedberg@vanderbilt.edu}

\begin{document}
\maketitle

\begin{abstract}
If \( \mathcal{V} \) is a finitely generated variety such that \( \Thfin{\mathcal{V}} \) is decidable, we show that \( \mathcal{V} \) is residually finite, and in fact has a finite bound on the sizes of subdirectly irreducible algebras. This result generalizes known results which assumed that \( \mathcal{V} \) has modular congruence lattices. Our proof of the theorem in its full generality proceeds by showing that strongly solvable radicals of algebras in \( \mathcal{V} \) are strongly abelian.
\end{abstract}

Let \( \var{V} \) be a class of mathematical structures. It is frequently of mathematical interest to compare the ``complexity'' of \( \var{V} \) along different scales, as a way of assessing the interaction of, for example, algebraic properties with logical ones.

In this investigation, we take up a question of this type: given a class which is reasonably simple from an algebraic perspective, how hard can algorithmic questions about this class be? To be more specific: we will consider classes \( \var{V} \) generated from a finite set of finite structures by applying a few standard algebraic operations, and ask whether either the class of all structures so generated (which is known to be computably axiomatizable) or the subclass consisting of just the finite structures, has decidable first-order theory. 

This problem in its full generality is still open, even assuming restrictive structural conditions on \( \var{V} \); however, we present here a surprisingly strong necessary condition for decidability of \( \Th{\var{V}_\mathrm{fin}} \), which eliminates a number of simplifying hypotheses from results obtained in the 1980s and 1990s.

We assume little background knowledge on the part of the reader; a nodding acquaintance with model- and computability-theoretic ideas and notations as might be encountered in a first graduate course in logic is the only true prerequisite. Our notation mostly follows contemporary texts like \cite{Rothmaler00}, \cite{Bergman11}.

The plan of the paper is as follows: In Section \ref{sec:Overview}, we lay out the notations and definitions needed for the investigation, after which, in Section \ref{sec:Results}, we can state our three main theorems. In Section \ref{sec:Comparable}, we prove Theorem \ref{MainThm:sigma comparable} and an important corollary (Corollary \ref{cor:sigma meet-irreducible}). Section \ref{sec:Abelian}, where Theorem \ref{MainThm:sigma strongly abelian} is proved, is the most difficult reading; by comparison, the proof of residual finiteness (Theorem \ref{MainThm:residual bound}) in Section \ref{sec:ResidualBound} is quick and transparent.

\section{Overview of fundamentals and preliminaries}\label{sec:Overview}
\setcounter{thm}{0}

\renewcommand{\theenumi}{\arabic{section}.\arabic{thm}.\roman{enumi}}

\subsection{Logical and algebraic definitions}

An \emph{algebra} is a first-order structure \( \alg{A} = \langle A; \cdots \rangle \) in a first-order language \( \mathcal{L} \) containing only function (and constant) symbols. The lattice of congruences of \( \alg{A} \) will be denoted \( \Con{\alg{A}} \). If \( \Con{\alg{A}} \) has a least nontrivial congruence \(\mu\), we call \( \alg{A} \) \emph{subdirectly irreducible} and \(\mu\) its \emph{monolith}. More generally, minimal nontrivial congruences are called \emph{atoms}.

A \emph{term operation} of \( \alg{A} \) is any finitary function \[ x_1, \ldots, x_n \mapsto t^{\alg{A}}(x_1, \ldots, x_n)\]on \(A\), for some \(\mathcal{L}\)-term \( t(v_1, \ldots, v_n) \). A \emph{polynomial operation} is a function \[ x_1, \ldots, x_k \mapsto t^{\alg{A}}(x_1, \ldots, x_k, a_{k+1}, \ldots, a_n)\]for some \( \mathcal{L}\)-term \(t\) and some elements \(a_i \in A\). The set of all polynomial operations of \(k\) or fewer variables is denoted \( \Pol{k}{\alg{A}} \). Unless otherwise specified, all first-order languages in this paper have only finitely many basic symbols, all of which are operations (or constants). (An important exception to this rule is the \emph{non-indexed algebras} described on page \pageref{defn:TCT}.)

The \emph{theory} of a first-order structure \( \alg{A} \) is the set of all \( \mathcal{L} \)-sentences true in \( \alg{A} \). If \( \mathcal{K} \) is a class of \( \mathcal{L} \)-structures, \( \Th{\mathcal{K}} \) is the set of all sentences true in all members of \( \mathcal{K} \). We write \( \mathcal{K}_\mathrm{fin} \) for the class of all finite members of \( \mathcal{K} \) and \( \Thfin{\mathcal{K}} \) for \( \Th{ \mathcal{K}_\mathrm{fin}} \). A class \(\mathcal{K} \) of \( \mathcal{L} \)-algebras is a \emph{variety} if it is axiomatized by some set of \emph{equations}, that is, sentences of the form \[ \forall \vec{v} \; t_1(\vec{v}) = t_2(\vec{v}) \]for some terms of the language. Equivalently, and more usefully for us, \( \mathcal{K} \) is a variety iff it is closed under taking direct products, subalgebras, and surjective homomorphic images. (Cf \cite{McKenzie_McNulty_Taylor}, \cite{Burris_Sanka}.) For a given algebra \( \alg{A} \) (resp. class \( \mathcal{K} \) of algebras) we denote the smallest variety containing it by \( \HSP{\alg{A}} \) (resp. \( \HSP{\mathcal{K}} \)).

If \( \var{V} \) is a variety and \( \kappa \) any cardinal, then \( \var{V} \) contains a free algebra on \( \kappa \) generators. If \( \kappa < \omega\), elements of this algebra are in canonical bijection (up to \(\Th{\var{V}}\)-equivalence) with \( \mathcal{L} \)-terms in \(\kappa\) variables.

\( \mathcal{V} \) is said to be \emph{residually \(\kappa\)} if for each \( \alg{A} \in \mathcal{V} \) and each \( a \neq b \in A \) there exists a homomorphism from \( \alg{A} \) onto a some algebra \( \alg{B} \) with \( \card{B} < \kappa\), separating \(a\) from \(b\). ``Residually \( \omega \)'' is usually called ``residually finite''. A \emph{residual bound} for \( \mathcal{V} \) is any cardinal \( \kappa \) such that \( \mathcal{V} \) is residually \( \kappa \). If every finitely generated \( \alg{A} \in \var{V} \) is finite, we say \( \var{V} \) is \emph{locally finite}.

For a given finite structure \( \alg{A} \), it is a trivial matter to determine whether a given first-order sentence holds in \( \alg{A} \); the same is not true for the problem of determining whether that same sentence holds throughout some variety containing \( \alg{A} \), such as \( \HSP{\alg{A}} \).

\begin{fact}\label{fact:defnFD}Let \( \alg{A} \) be any finite algebra, \( \mathcal{V} = \HSP{\alg{A}} \).
  \begin{enumerate}
    \item \( \mathcal{V} \) is locally finite and computably axiomatizable; it follows that \( \Th{\mathcal{V}} \) is computably enumerable. We will say that \( \mathcal{V} \) is \emph{decidable} if this set of sentences is computable.
    \item The \emph{complement} of \( \Thfin{\mathcal{V}} \), the set of all sentences \emph{falsified} in some finite member of \( \mathcal{V} \), is computably enumerable. We will say that \( \mathcal{V} \) is \emph{finitely decidable} if this set of sentences is computable.
  \end{enumerate}
\end{fact}

There do exist finite algebras \( \alg{A} \) such that \( \HSP{\alg{A}} \) is undecidable and/or finitely undecidable. For example, by \cite{Zamyatin78}, any non-abelian finite group generates an undecidable variety; for many other instances of undecidable and/or finitely undecidable varieties, see \cite{Malcev60}, \cite{Ershov72}, \cite{Zamyatin76}, \cite{KIdziak86}, \cite{McKenzie_Valeriote}, \cite{K+PMIdziak88}, \cite{Idziak89i}, \cite{Idziak89ii}, and \cite{Jeong99}.

As the alert reader has seen in Fact \ref{fact:defnFD}, there is a fundamental asymmetry between decidability and finite decidability, as in the one case it is the set of provable sentences which is easily shown to be enumerable, while in the other it is the refutable sentences. This asymmetry is not just apparent: the two properties are in fact completely independent. Specific examples of the four possibilities are given in \cite{Szmielew55}, \cite{K+PMIdziak88}, \cite{Olshanskii91}, and \cite{Jeong99}.

The principal tool this investigation will employ in establishing undecidability is the method of interpretation. Very briefly, we will repeatedly establish that some class of structures \( \mathcal{K}\) in some first-order language is not finitely decidable by finding a ``uniformly definably isomorphic'' copy of a class \( \mathcal{G}_{\mathrm{fin}}\) in \(\mathcal{K}_{\mathrm{fin}}\), where \( \mathcal{G} \) is finitely axiomatizable and finitely undecidable. The reader is referred to standard texts \cite{Hodges} Chapter 5, \cite{Burris_Sanka} Section V.5, for more details. Observe that if an undecidable class \( \mathcal{G}_{\mathrm{fin}} \) interprets into \( \mathcal{K}_{\mathrm{fin}}\) as above, then not only \( \mathcal{K}\) but every class \(\mathcal{K}' \supset \mathcal{K}\) of structures in the language is finitely undecidable as well: we say that \( \mathcal{K} \) is \emph{hereditarily} finitely undecidable.

 The classes we will be interpreting will be the class of graphs and the class \( \EQtwo \), defined below. For this investigation, a \emph{graph} is a first-order structure \( \mathbb{G} = \langle V; E \rangle \), where \(E^{\mathbb{G}}\) is a symmetric, irreflexive binary relation. (It follows that graphs in our sense do not possess multiple edges between a single pair of vertices.) It was shown by Ershov and Rabin in the 1960s that graphs are both undecidable and finitely undecidable.

\( \EQtwo \)\label{defn:two equivalence relations} is the class of structures \( \alg{E} = \langle I; R_0, R_1 \rangle \) where each \(R_i \) is a binary predicate symbol whose interpretation in the structure is an equivalence relation on \(I\), such that \( R_0^{\alg{E}} \cap R_1^{\alg{E}} = \bot_I \). We will sometimes refer to \( \Th{\EQtwo} \) as the theory of two disjoint equivalence relations. Corollary 5.16 of \cite{Burris_Sanka} shows that the theory of this class is undecidable and finitely undecidable.

(In fact, it can be shown that for each of the above classes, \( \Th{\mathcal{K}} \) is \emph{computably inseparable} from the set of sentences finitely refutable in \( \mathcal{K} \); but we will not need this stronger property.)

\subsection{Abelian and solvable algebras and TCT}
Modern investigations in universal algebra are greatly aided by the linked toolboxes of the theory of solvable and strongly solvable algebras and congruences (see for example \cite{Freese_McKenzie}) and the ``tame congruence theory'' developed by the first author and David Hobby in \cite{Hobby_McKenzie}.

Let \( \alg{A} \) be any algebra, and \( \alpha, \beta, \gamma \) be congruences (or more generally, any binary relations) on \( A \). \( \alg{A} \) is said to satisfy the \emph{term condition} \( \TC{\alpha}{\beta}{\gamma} \) if the implication \begin{align*}
       t(\vec{a}_1,\vec{b}_1) &\equiv_\gamma t(\vec{a}_1, \vec{b}_2) \\ & \Downarrow \\
       t(\vec{a}_2,\vec{b}_1) &\equiv_\gamma t(\vec{a}_2, \vec{b}_2)
     \end{align*}is valid for all terms \(t\) and all tuples \( \vec{a}_1 \equiv_\alpha \vec{a}_2 \) and \( \vec{b}_1 \equiv_\beta \vec{b}_2 \). If \(R,S \subset A\), then we will write \( \TC{R}{S}{\gamma} \) when we mean \( \TC{R^2}{S^2}{\gamma} \). If \( \gamma \leq \beta \in \Con{\alg{A}}\) and \( \TC{\beta}{\beta}{\gamma} \), then we say that \( \beta \) is \emph{abelian} over \( \gamma \). If \( \TC{\beta}{\beta}{\bot_A} \) then we say that \( \beta \) is an abelian congruence. If \( \TC{\top_A}{\top_A}{\bot_A} \) then we say that \( \alg{A} \) is an abelian algebra.

\stepcounter{thm}
We can always transform a failure \begin{align*}
  t(\vec{a}_1, \vec{b}_1) &= t(\vec{a}_1, \vec{b}_2) \\ &\text{but} \\
  t(\vec{a}_2, \vec{b}_1) &\neq t(\vec{a}_2, \vec{b}_2)
\end{align*}of \( \TC{\alpha}{\beta}{\gamma} \) into one \begin{align*}
  s(a_1', \vec{b}_1') &= s(a_1', \vec{b}_2') \\ &\text{but} \tag{\thethm}\label{eq:TCsimplification}\\
  s(a_2', \vec{b}_1') &\neq s(a_2', \vec{b}_2')
\end{align*}where \(\alpha\)-shifting occurs in only one variable. The same is not true in general for the \(\beta\)-shifted variables; however, this is possible in the special case where all the elements in \( \vec{b}_1, \vec{b}_2\) are taken from some \(U \subset A \) such that every operation on \(U\) is realized by a polynomial of \( \alg{A} \). We leave the verification of this to the reader.

Another asymmetry between the roles played by the first two variables of the term condition has to do with congruence generation. If \(R\) is a binary relation on \(A\), then \( \TC{R}{\beta}{\gamma} \) holds iff \( \TC{\rho}{\beta}{\gamma} \) does, where \( \rho \) is the least congruence of \( \alg{A} \) identifying all the pairs in \(R \cup \gamma\). By comparison, \( \TC{\alpha}{R}{\gamma} \) holds iff \( \TC{\alpha}{\alg{S}}{\gamma} \), where \( \alg{S} \) is the reflexive, symmetric subalgebra of \( \alg{A}^2 \) generated by \(R\).

  If \( \gamma \leq \beta \in \Con{\alg{A}}\), we say that \( \beta \) satisfies the \emph{strong} term condition over \( \gamma \), or that \( \beta \) is \emph{strongly abelian} over \( \gamma \), if for all terms \(t\) and tuples \( \vec{a}_1 \equiv_\beta \vec{a}_2 \), \( \vec{b}_1 \equiv_\beta \vec{b}_2 \equiv_\beta \vec{b}_3 \), \begin{align*}
    t(\vec{a}_1, \vec{b}_1) &\equiv_\gamma t(\vec{a}_2, \vec{b}_2) \\ &\Downarrow \\
    t(\vec{a}_1, \vec{b}_3) &\equiv_\gamma t(\vec{a}_2, \vec{b}_3)
  \end{align*}If \( \TC{\beta}{\beta}{\gamma} \), this condition is equivalent to the apparently weaker condition \[ \vec{a}_1 \equiv_\beta \vec{a}_2 \: \& \: \vec{b}_1 \equiv_\beta \vec{b}_2 \: \& \: t(\vec{a}_1, \vec{b}_1) \equiv_\gamma t(\vec{a}_2, \vec{b}_2) \; \Rightarrow \; \forall i, j \; t(\vec{a}_1, \vec{b}_1) \equiv_\gamma t(\vec{a}_i, \vec{b}_j) \]which is easier to use.

If \( \alg{A} \) is a locally finite algebra and \( \alpha^- < \alpha^+ \in \Con{\alg{A}} \), we say that \( \alpha^+ \) is (strongly) \emph{solvable} over \( \alpha^- \) if every chain of congruences \[ \alpha^- = \beta_0 < \beta_1 < \cdots < \beta_{m-1} < \beta_m = \alpha^+ \]admits a refinement \[ \alpha^- = \gamma_0 < \gamma_1 < \cdots < \gamma_{n-1} < \gamma_n = \alpha^+ \]such that each \( \gamma_{i+1} \) is (strongly) abelian over \( \gamma_i \).

Let \( \alg{A} \) be a finite algebra and \( \alpha \prec \beta \) in \( \Con{\alg{A}} \) (that is, \( \beta \) is an upper cover of \( \alpha \) in the order-theoretic sense). For any subset \( W \subset A \), the \emph{non-indexed algebra}\label{defn:TCT} \( \alg{A}_{|W} \) \emph{induced by \(\alg{A}\) on \(W\)} is defined to have underlying set \(W\), and a basic operation \( f(v_1, \ldots, t_k) \) for each polynomial \(f \in \Pol{k}{\alg{A}} \) such that \( f(W^k) \subset W \). We do not usually wish to specify any more parsimonious signature for an induced algebra; even if the signature of \( \alg{A} \) was finite, \( \alg{A}_{|W} \) is not in general representable as a first-order structure in any finite language.

An \((\alpha,\beta)\)-\emph{minimal set} \(U \subset A\) is an inclusion-minimal polynomial image \(e(A)\) of the algebra, where \(e \in \Pol{1}{\alg{A}} \) is required to be \emph{idempotent} (\(e \circ e = e\)) and to preserve the \( \alpha \)-inequivalence of some pair \( \langle a,b \rangle \in \beta \setminus \alpha\). Clearly, every \((\alpha,\beta)\) minimal set has at least two elements. If \(U\) is \((\alpha,\beta)\)-minimal, a \(\beta_{|U}\)-class which properly contains two or more \( \alpha_{|U}\)-classes is called a \emph{trace}. The union of the traces included in \(U\) is called the \emph{body} of \(U\); the remainder is called the \emph{tail}.

\begin{thm}[(Fundamental Theorem of Tame Congruence Theory, \cite{Hobby_McKenzie} Theorem 2.8, Theorem 4.7, Lemma 4.8)]Let \( \alg{A} \) be a finite algebra with congruences \( \alpha \prec \beta \).
  \begin{enumerate}
    \item \label{enum:TCTtype wel defined} All \((\alpha,\beta)\)-minimal sets \(U_1, U_2\) are polynomially isomorphic, in the sense that there exists \(f \in \Pol{1}{\alg{A}} \) which maps \(U_1\) bijectively to \(U_2\) in such a way that every induced operation \[ t_2 \in U_2^{U_2^k} \]in the signature of \( \alg{A}_{|U_2}\) is the \(f\)-image of an operation \[ t_1 \in U_1^{U_1^k} \]in the signature of \( \alg{A}_{|U_1}\).
    \item Let \(N \subset U\) be any trace in an \((\alpha, \beta)\)-minimal set. If \( \alg{A}_{|N}/\alpha_{|N} \) is isomorphic to the two-element boolean algebra, the two-element lattice, or the two-element semilattice, then we say that the covering is of (respectively) boolean type (\(\alpha \prectype{3} \beta \)), lattice type (\( \alpha \prectype{4} \beta\)), or semilattice type (\( \alpha \prectype{5} \beta \)). (This is well-defined by (\ref{enum:TCTtype wel defined}).)
    \item If none of these possibilities occur, then \( \alg{A}_{|N}/\alpha_{|N} \) is an abelian algebra, and is either isomorphic to a finite module over some ring, in which case the cover is of affine type (\(\alpha \prectype{2} \beta\)); or isomorphic to a finite \(G\)-set for some finite group \(G\) (unary type, \(\alpha \prectype{1} \beta \)). In the former case, \( \beta \) is abelian over \( \alpha \) but not strongly abelian; in the latter, \( \beta \) is strongly abelian over \( \alpha \).
  \end{enumerate}
\end{thm}

We will write \( \typset{\alg{A}} \subset \{1,2,3,4,5\}\) for the set of tame congruence types which appear in \( \Con{\alg{A}} \).

Let \(i \neq j \) be tame congruence types. We will say that the algebra \( \alg{A} \) satisfies the \emph{\( (i,j) \)-transfer principle} if, for all covering chains \[ \alpha_1 \prectype{i} \alpha_2 \prectype{j} \alpha_3 \]there exists \[ \alpha_1 \prectype{j} \beta_j \leq \alpha_3 \]and likewise \[ \alpha_1 \leq \beta_i \prectype{i} \alpha_3 \]

\begin{fact}\label{fact:basics}Let \( \var{V} \) be a finitely decidable variety.
  \begin{enumerate}
    \item\label{enum:basics45} \( \var{V} \) omits the lattice and semilattice tame congruence types.
    \item\label{enum:transfer} The (1,2), (2,1), (3,1), and (3,2) transfer principles hold throughout \(\var{V}\); in particular,
    \item If \( \alg{S} \in \var{V} \) is a finite subdirectly irreducible algebra with boolean-type monolith, then \( \typset{\alg{S}} = \{3\} \). If the monolith is affine, then \( \typset{\alg{S}} \subset \{2,3\} \), and if the monolith is unary, then \( \typset{\alg{S}} \subset \{1,3\} \).
    \item \label{enum:emptytail} If \( \alg{A} \in \var{V}\) and \( \alpha \prectype{2,3} \beta \), then all \( (\alpha,\beta)\)-minimal sets have no tail. In the boolean case, this means that each minimal set contains just two elements, and every possible operation from this set to itself is realized by a polynomial of the algebra.
  \end{enumerate}
\end{fact}

\begin{proof}
  \eqref{enum:basics45} is proved in \cite{Hobby_McKenzie} Theorem 11.1; it is a consequence of the fact that (finite) graphs interpret semantically into each of \[ \HSP{\langle \{0,1\}; \land \rangle} \]and \[ \HSP{\langle \{0,1\}; \land, \lor \rangle} \]\eqref{enum:transfer} is proved in \cite{Valeriote_Willard} and \cite{Valeriote94}. \eqref{enum:emptytail} is also proved in \cite{Valeriote_Willard}.
\end{proof}

It follows by Theorem 8.5 of \cite{Hobby_McKenzie} that any locally finite, finitely decidable variety omitting the unary type is congruence-modular.

The following fact will be of use later in the paper:

\begin{thm}[\cite{Hobby_McKenzie} Chapter 7]\label{thm:SSsim is congruence}Let \( \alg{A} \) be any finite algebra.
  \begin{enumerate}
    \item Each of the relations
      \[ \alpha \stackrel{ss}{\sim} \beta \iff \alpha \text{ is connected to } \beta \text{ via covers of type 1} \]and
      \[ \alpha \stackrel{s}{\sim} \beta \iff \alpha \text{ is connected to } \beta \text{ via covers of types 1 and 2} \]is a lattice congruence of \( \Con{\alg{A}} \).

    \item If \( \alpha \leq \beta \) and \( \gamma \in \Con{\alg{A}} \) is any other congruence, and if the interval from \( \alpha \) to \( \beta \) contains only covers of type 1, then the same is true for each of the intervals \( \gamma \land \alpha \leq \gamma \land \beta\), \( \gamma \lor \alpha \leq \gamma \lor \beta \). 
  \end{enumerate}
\end{thm}

It follows that for every finite algebra \( \alg{A} \), the sets of congruences \( \stackrel{ss}{\sim} \)-equivalent (resp. \( \stackrel{s}{\sim} \)-equivalent) to \( \bot_A \) have largest elements, which we call the strongly solvable radical \( \SSRad{\alg{A}} \) and solvable radical \( \Rad{\alg{A}} \) of \( \alg{A} \).

\subsection{Powers and Subpowers}
Let \(\alg{A}\) be any algebra. A \emph{subpower} of \(\alg{A}\) is a subalgebra \( \alg{B} \leq \alg{A}^I\) for some index set \(I\). We use two notations for elements of powers and subpowers: the element with coordinate \(x^i\) at place \(i \in I\) may be denoted \( \aI{x} = \langle x^i \rangle_{i \in I} \); alternatively, if only a few elements \(a_1, a_2, \ldots\) of \(A\) appear as coordinates of \(\aI{x}\), we may instead use a direct sum notation \[ \aI{x} = {a_1}_{|I_1} \oplus {a_2}_{|I_2} \oplus \cdots\](where \(I_j\) is the set of indices where \(a_j\) appears). \( \alg{B} \leq \alg{A}^I \) is said to be \begin{itemize}
  \item \emph{subdirect} (notation: \(\alg{B} \leq_s \prod \alg{A}^I \)) if for each \(i \in I \) and each \(a \in A\) there exists \( \aI{x} \in B \) with \( x^i = a\), and
  \item\label{defn:diagonal subpower} \emph{diagonal} if for each \(a \in A\), the element \( \aI{a} = \langle a \rangle_{i \in I} \) belongs to \( B \). We will freely identify \(\alg{A}\) with its image under the diagonal embedding.
\end{itemize}

If \( \alg{A} \) is an algebra, \( U \subset A \), and \( \alg{B} \leq \alg{A}^I \), we will frequently be interested in subsets of the form \( U^I \cap B \). If the meaning is clear from context, we will usually abbreviate this to \( U^I \).

\begin{prop}\label{prop:power of minimal set}
  Let \( \alg{A} \) be any algebra, and let \( e \in \Pol{1}{\alg{A}} \) be idempotent (that is, \( e \circ e = e \)). Then if \( U = e(A) \), and if \( \alg{B} \leq \alg{A}^I\) is any diagonal subpower of \( \alg{A} \), then \( U^I \cap B \) is an \( A \)-definable subset of \(\alg{B}\).
\end{prop}

\begin{proof}
  Since \( \alg{B} \) contains the diagonal, the function \(\aI{e} = e^I \) is realized as a polynomial of \( \alg{B} \). \( U^I \cap B \) is the set of fixed points of this polynomial.
\end{proof}

Indeed, for any such diagonal subpower and for each \(k\), the map \begin{align*} \Pol{k}{\alg{A}} &\hookrightarrow \Pol{k}{\alg{B}} \\ f(v_1, \ldots, f_k) = t(v_1, \ldots, v_k, a_1, \ldots, a_\ell) &\mapsto t(v_1, \ldots, v_k, \aI{a}_1, \ldots, \aI{a}_\ell) = f^I \end{align*}is an embedding (of clones), which we will make continual use of.

\begin{lemma}\label{lemma:subpowers omit type 1}
  Let \( \alg{A}_i \), \(1 \leq i \leq p \) be finite algebras with trivial strongly solvable radical. Then every \[ \alg{B} \leq_s \prod_i \alg{A}_i \]has trivial strongly solvable radical.
\end{lemma}

\begin{proof}
  We show the contrapositive: suppose that \( \bot_B \prectype{1} \alpha \) is an atom of \( \Con{\alg{B}} \). Then there is some projection congruence \( \eta_j \) such that \( \alpha \lor \eta_j > \eta_j \). By Theorem \ref{thm:SSsim is congruence}, since \( \bot_B \stackrel{ss}{\sim} \alpha \), \( \eta_j \stackrel{ss}{\sim} \alpha \lor \eta_j \); it follows that the strongly solvable radical of \( \alg{A}_j \) sits above \( \alpha \lor \eta_j \).
\end{proof}

  If \( \alg{A}_1, \ldots, \alg{A}_p, \alg{B} \) are as in the previous Lemma, and all belong to some finitely decidable variety, then we can conclude (via the transfer principles) that in fact \( \alg{B} \) has no unary-type covers anywhere in its congruence lattice. This remains true if we introduce finitely many constant symbols in such a way that each element of each \( A_i \) is named by at least one constant symbol; call these expansions \( \langle \alg{A}_i; A_i \rangle \). Lemma \ref{lemma:subpowers omit type 1} implies that \( \HSP{\{ \langle \alg{A}_i; A_i\rangle \}_{i=1}^p} \) is modular (since all minimal sets will have empty tails), and so has Day (or Gumm) terms.

  In particular, if we are considering a fixed finite \( \alg{B} \leq_s \prod_i \alg{A}_i \), we may introduce constant symbols for each element of \( B \) and interpret them in the \( \alg{A}_i \) via their coordinate projections. Then \( \langle \alg{B}; B \rangle \) has Day terms, which become Day polynomials when we reduct back out to the original language. It follows that all the nice properties of congruence-modular varieties, such as most of the theory of commutators, hold for \( \alg{B} \).

It is an open problem whether the finite decidability of \[ \HSP{\alg{A}_1, \ldots, \alg{A}_p} \] implies the finite decidability of \[ \HSP{\langle \alg{A}_1; A_1 \rangle, \ldots, \langle \alg{A}_p; A_p \rangle} \]The best we can say is that the latter variety must be \(\omega\)-structured, in the sense of \cite{McKenzie_Valeriote}.

\subsection{Main Results}\label{sec:Results}
Let \( \var{V} \) be a fixed locally finite, finitely decidable variety. In \cite{Idziak97}, it is shown that

\begin{thm}\label{thm:type2_abelian_centralizer}
  If \(\alg{S}\) is a subdirectly irreducible algebra in \(\var{V}\) with monolith \(\mu\) such that \(\typ{\bot}{\mu} = 2\), then the centralizer of \(\mu\) is an abelian congruence, and is in fact the solvable radical \(\Rad{\alg{S}}\). Moreover, every congruence of \( \alg{S} \) is comparable to \(\Rad{\alg{S}}\).
\end{thm}

Our first two main theorems generalize this result:

\begin{MainThm}\label{MainThm:sigma comparable}
  If \( \alg{S} \in \var{V} \) is a finite subdirectly irreducible algebra, then the strongly solvable radical \( \SSRad{\alg{S}} \) is comparable to every congruence on \( \alg{S} \).
\end{MainThm}

\begin{MainThm}\label{MainThm:sigma strongly abelian}
  If \( \alg{S} \in \var{V} \) is a finite subdirectly irreducible algebra with unary-type monolith \(\mu\), then the centralizer of \(\mu\) equals \( \SSRad{\alg{S}} \) (which must also be \( \Rad{\alg{S}} \)), and is a strongly abelian congruence.
\end{MainThm}

The third main theorem gives a strong structural critereon for finite decidability:

\begin{MainThm}\label{MainThm:residual bound}
  If \(\var{V}\) is any finitely decidable, finitely generated variety, then \( \var{V} \) has a finite residual bound; in particular, all algebras in \( \var{V} \) are residually finite.
\end{MainThm}

\section{On comparability of the strongly solvable radical}\label{sec:Comparable}
\setcounter{thm}{0}

The first section will be devoted to proving Theorem \ref{MainThm:sigma comparable}. We begin with two old facts that will be useful.

\begin{lemma}\label{lemma:zeta solvable}
  Let \( \alg{S} \) be a subdirectly irreducible algebra in a finitely decidable variety with unary-type monolith \(\mu\). Then the centralizer of \(\mu\), the greatest congruence \( \zeta \) such that \( \TC{\zeta}{\mu}{\bot} \), is strongly solvable.
\end{lemma}

\begin{proof}
	This is proved in \cite{Idziak_Valeriote01}, Theorem 4.
\end{proof}

\begin{lemma}\label{lemma:Maroti} Let \( \alg{A} \) be a finite algebra with \( \bot_A \prectype{1} \delta \), and let \(U \) be \((\bot,\delta)\)-minimal.
  \begin{enumerate}
  \item If \( D_1, \ldots, D_k \) are \(\delta\)-classes, then every mapping \[ f: D_1 \times \cdots \times D_k \rightarrow U\](where \(f \in \Pol{k}{\alg{A}} \)) depends on no more than one of its variables.
  \item {\rm (Maroti's Lemma)} If \(\delta \leq \beta\) in \(\Con{\alg{A}}\) and \(\TC{\beta}{\delta_{|U}}{\bot}\), and \( B_1, \ldots, B_k \) are \(\beta \)-classes, then for every mapping \[ f: B_1 \times B_2 \times \cdots \times B_k \rightarrow U \](where \(f \in \Pol{k}{\alg{A}} \)) there exists \(1 \leq j \leq k\) so that \[\vec{x} \equiv_\delta \vec{y}\text{ and }x_j = y_j \Rightarrow f(\vec{x}) = f(\vec{y})\]
  \end{enumerate}
\end{lemma}

\begin{proof}
  The second statement is Lemma 7.2 of \cite{IMV_Few_Models}; the first statement is a special case of the second (or can be proved independently, as in \cite{Hobby_McKenzie} Theorem 5.6).
\end{proof}

\begin{defn}[(\cite{Kearnes93} Definition 4.1)]\label{defn:coherent}
  Let \( \alpha \prec \beta \) be a congruence cover of the finite algebra \( \alg{A} \), and let \( \gamma \in \Con{\alg{A}} \). Let \(T\) denote the set of all \( (\alpha,\beta) \)-traces in \( \alg{A} \). We say that \( (\alpha,\beta) \) is \emph{\( \gamma \)-coherent} if \[ \bigwith_{N \in T} \TC{\gamma}{\beta_{|N}}{\alpha} \quad \Longrightarrow \quad \TC{\gamma}{\beta}{\alpha} \]If \( \alpha = \bot \) then we will say that \( \beta \) is \(\gamma\)-coherent. Note, that since all \( (\alpha,\beta) \)-traces are polynomially isomorphic, \( \TC{\gamma}{\beta_{|N}}{\alpha} \) holds for all \(N \in T \) iff it holds for any such \(N\).
\end{defn}

Our first technical lemma has nothing particular to do with decidability:
\begin{lemma}\label{lemma:binary witnesses}
  Let \( \alg{A} \) be any finite algebra with congruences \( \bot \prectype{1} \delta \) and \( \alpha \prectype{3} \beta \), such that \( \beta = \Cg{}{0}{1}\) for some (hence any) \( (\alpha,\beta) \)-trace \( \{0,1\} \). Assume further that \( \neg \TC{\beta}{\delta}{\bot}\). Then there exists a polynomial \(p(x,y) = p(x,p(x,y)) \) taking values in some \( (\bot,\delta) \)-minimal set \(U\), so that \begin{enumerate}
    \item If \( \delta \) is \( \beta \)-coherent, then \( p(0,y) \) collapses traces to points and \( p(1,u) = u \) for all \( u \in U \);
    \item If \( \delta \) is \( \beta \)-incoherent, then \(p(0,u) = u = p(1,u) \) for all \(u \in U \), but for some \( c \in U \), \( c \equiv_\delta d \), \(d \notin U \), \begin{align*}
      p(0,c) &= p(0,d) \\ &\text{but} \\
      p(1,c) &\neq p(1,d)
    \end{align*}witnesses the failure of centralization.
\end{enumerate}
\end{lemma}

\begin{proof}
  Suppose first that \( \delta \) is \(\beta\)-coherent. Then for some \( (\bot,\delta)\)-trace \(N\) included in some minimal set \(U\), we have \( \neg \TC{\beta}{N}{\bot} \). Since \(\beta\) is generated by \(\{0,1\}\), \( \TC{\{0,1\}}{N}{\bot} \) must already be false.

  Choose a witnessing package \begin{align*}
    t(0,\vec{c}) &= t(0,\vec{d}) \\ &\text{but} \\
    t(1,\vec{c}) &\neq t(1,\vec{d})
  \end{align*}where we may choose \(t\) so that its range lies entirely in \(U\). The polynomial mapping \(t(1,\vec{y})\) is essentially unary as a mapping from \(\vec{N} \) into \(U\); say it depends on \(y_1\), and let \( p(x,y) = t(x,y,c_2, c_3, \dots) \). Then \(p(1,c_1) = t(1,\vec{c}) \neq t(1,\vec{d}) = p(1,d_1) \) while \( p(0,c_1) = p(0,d_1) \). Iterating \(p\) in the second variable if necessary, we get a polynomial satisfying the Lemma.

  The other case requires a bit more work.

  Assume now that for all traces \(N\), we have \( \TC{\beta}{N}{\bot} \). As in the first case, \( \neg \TC{\beta}{\delta}{\bot} \) implies that \( \neg \TC{\{0,1\}}{\delta}{\bot} \) already. Take a witnessing package \begin{align*}
    t(0,\vec{c}) &= t(0,\vec{d}) \\ &\text{but} \\
    t(1,\vec{c}) &\neq t(1,\vec{d})
  \end{align*}where we may assume that the image of \(t\) is contained in some \((\bot,\delta)\)-minimal set \(U_0\). The map \( t(0,\vec{y}): {c_1}/\delta \times {c_2}/\delta \times \dots \rightarrow U_0 \) depends only one one variable, say \(y_{k_0}\), and likewise \(t(1,\vec{y})\) on \(y_{k_1}\).

  \begin{claim}
    \(k_0 = k_1\)
  \end{claim}

  Suppose the Claim were false. Let \(q(x,y) = t(x,c_1, \dots, c_{k_1 - 1}, y, c_{k_1 + 1}, \dots )\). Then \(q(0,c_{k_1}) = q(0,y) \) for all \(y \equiv_\delta c_1\).

\addtocounter{claim}{1}
  Now, since \( c_{k_1} \equiv_\delta d_{k_1} \), there exists a sequence \[\tag{\EQnum} \label{eq:walk from ck to dk}c_{k_1} = a_0, a_1, \dots, a_\ell = d_{k_1}\]where each pair \( \{a_i, a_{i+1}\} \) belong to a \( (\bot,\delta) \)-trace \(N_i\) (\(i < \ell\)) included in a minimal set \(U_i = e_i(A)\). Since \(q(1,a_0) \neq q(1,a_{\ell}) \), there must exist some \(i < \ell\) such that \(q(1,a_i) \neq q(1,a_{i + 1}) \). But we have already seen that \(q(0,a_i) = q(0,a_{i+1}) \), contradicting \( \TC{\{0,1\}}{\delta_{|N_i}}{\bot} \). This proves the Claim, and we may set \(k := k_0 = k_1 \).

  Let \(a_0, a_1, \dots, a_{\ell} \) be the sequence defined in (\ref{eq:walk from ck to dk}); our as\-sump\-tion that \( \{0,1\} \) centralizes \(N_i\) means that for each \(i < \ell\), \(q(0,y)\) is injective on \(N_i\) iff \(p(1,y)\) is.

  Let \(i\) be the first index for which \(q(1,a_i) \neq q(1,a_{i+1})\); then \begin{align*}
    q(0,d_k) &=    q(0,c_k) = q(0,a_0) = q(0,a_1) = \ldots = q(0,a_i) \\ &\text{but} \\
    q(1,d_k) &\neq q(1,c_k) = q(1,a_0) = q(1,a_1) = \ldots = q(1,a_i)
  \end{align*}Then with \(c = a_i \), \(d = d_k\), \(U = U_i\), and \(p(v_0,v_1)\) equalling an iterate of \(e_i \circ q(v_0,v_1) \) such that \(p(x,p(x,y)) = p(x,y)\) for all \(x,y \in A\), the conclusions of the Lemma are satisfied.
\end{proof}

\begin{lemma}\label{lemma:!C(K,N) or !C(N,K)}
  If \( \bot_{\alg{A}} \prectype{1} \delta \), \( \alpha \prectype{3} \beta \), \( K = \{0,1\} \), and \(N \subset U\) are as in the statement of Lemma \ref{lemma:binary witnesses}, then at least one of \( \TC{K}{N}{\bot} \) and \( \TC{N}{K}{\bot} \) must fail. In both cases, the failure is witnessed by a binary polynomial which takes \( K \times N \) into \(U\).
\end{lemma}

\begin{proof}
  In the case where \( \delta \) is \(\beta\)-coherent, the polynomial \(p\) found in that Lemma witnesses \( \neg \TC{K}{N}{\bot} \).

  So let \( \TC{\beta}{N}{\bot} \) for all \( (\bot,\delta)\)-traces \(N\), and fix witnesses \begin{align*}
    c = p(0,c) &= p(0,d) \\ &\text{but} \\
    c = p(1,c) &\neq p(1,d)
  \end{align*}where \(c \in U\), \( c \equiv_\delta d \notin U\), and the range of \(p\) is contained in \(U\). We aim to show that \( \TC{N}{K}{\bot} \) fails, and that its failure is witnessed by a binary polynomial of the claimed kind.

  Let \[ c = a_0, a_1, \dots, a_{\ell - 1}, a_\ell = d \]be a walk from \(c\) to \(d\) through traces (see the discussion following Equation \eqref{eq:walk from ck to dk}).  Since \[p(0,a_i) \equiv_\delta p(0,a_0) = c = p(1,a_0) \equiv_\delta p(1,a_i) \]for all \(i \leq \ell\), we know that \(p(K,c/\delta) \subset N\). Now let \( \{a_j, a_{j+1}\} \subset N_j \) be the first step where \begin{align*}
    p(0,a_j)    &=    p(1,a_j) \\ &\text{but} \\
    p(0,a_{j+1}) &\neq p(1,a_{j+1})
  \end{align*}By hypothesis, \(j > 0\). It follows that at least one, and hence both, of \(p(0,y) \) and \(p(1,y)\) are polynomial isomorphisms from \(N_j\) to \(N\). Let \( q \in \Pol{1}{\alg{A}} \) be the inverse isomorphism to \( p(0,y) \), where \(q(a) = a_j \) and \( q(a') = a_{j+1} \). Then \begin{align*}
    p(0,q(a)) = p(0,a_j)    &=    p(1,a_j) = p(1,q(a)) \\ &\text{but} \\
    p(0,q(a')) = p(0,a_{j+1}) &\neq p(1,a_{j+1}) = p(1,q(a'))
  \end{align*}so that \(p(x,q(y))\) witnesses \( \neg \TC{N}{K}{\bot} \) as required.
\end{proof}

We are now ready to start generating undecidable problems:

\begin{lemma}\label{lemma:disjoint centralizer}
  Let \( \alg{A} \) be a finite algebra, \( \bot \prectype{1} \delta \) and \( \alpha \prectype{3} \beta \), and let \(K = \{0,1\} \) be \( (\alpha,\beta) \)-minimal, where \(\beta = \Cg{}{0}{1}\). If \( \neg \TC{K}{\delta}{\bot} \) and \( \TC{\delta}{K}{\bot} \), then \( \HSP{\alg{A}} \) is hereditarily finitely undecidable.
\end{lemma}

In other words, the centralizer of a boolean neighborhood must be disjoint from any of the unary-type atoms (or at least those which that neighborhood does not itself centralize), if \( \alg{A} \) is to live in a finitely decidable variety.

\begin{proof}
  Fix a \( (\bot,\mu) \)-minimal set \(U\). By Lemma \ref{lemma:!C(K,N) or !C(N,K)}, for any \((\bot,\delta)\)-trace \(N \subset U\), at least one of \(  \TC{K}{N}{\bot} \) or \( \TC{N}{K}{\bot} \) must fail. But if \( \neg \TC{N}{K}{\bot} \) then \( \neg \TC{\delta}{K}{\bot} \), contrary to the assumptions of the Lemma. 

Hence \( \neg \TC{K}{N}{\bot}\). Choose a witnessing package \begin{align*}
    q(0,c) &= q(0,d) \\ &\text{but} \\
    q(1,c) &\neq q(1,d)
  \end{align*}Without loss of generality, we can assume that \(q(1,u) = u \) for all \(u \in U\).

  Our plan is to semantically interpret the class of graphs with at least three vertices into diagonal subpowers of \( \alg{S} \). So let \( \mathbb{G} = \langle V, E \rangle \) be such a graph, and let \(I = V \times \{+,-\} = V^\pm \). Define \( \alg{D} = \alg{D}(\mathbb{G}) \leq \alg{A}^I\) to be generated by the diagonal together with the points \begin{align*}
    \chi^\beta_v &:= 1_{|\{v^+,v^-\}} \oplus 0_{|\mathrm{else}} \qquad (\text{all } v \in V) \\
    \chi^\delta_e &:= d_{|\{v^+,w^+\}} \oplus c_{|\mathrm{else}} \qquad (\text{all } e = \langle v,w \rangle \in E )\\
    \chi^\delta_{V^+} &:= d_{|V^+} \oplus c_{|V^-}
  \end{align*}Let \( \vec{\chi}^\beta \) and \( \vec{\chi}^\delta \) enumerate the respective sets of generators.
	
  Observe that there cannot be any nonconstant polynomial map from \(N\) to \( \{0,1\} \). This implies that \( D \cap \{0,1\}^I \) consists of all points which are constant on each set \( \{v^+, v^- \} \); in other words, \( \alg{D}_{|\{0,1\}^I} \) is canonically isomorphic to the boolean algebra \( \alg{2}^V \). This subset is definable (Proposition \ref{prop:power of minimal set}), as is its set of atoms \( \{ \chi^\beta_v \colon v \in V \} \); by abuse of language, we will allow ourselves to quantify over these atoms by saying things like ``there exists a vertex \( \chi^\beta_v \)...''

\begin{claim}\label{claim:edges (disjoint centralizer)}
  The set of those \( \aI{x} \in D \) of the form \( d_{|\{w_1^+,w_2^+\}} \oplus c_{|\mathrm{else}} \) for two distinct vertices \(w_1,w_2 \in V\) is definable (using the parameter \( \chi^\delta_{V^+} \)).
\end{claim}

It is sufficient to show that for \( \aI{x} \in U^I \cap D \), \begin{align}
\stepcounter{claim}		
  q \left( \chi^\beta_{w_1} + \chi^\beta_{w_2}, \aI{x} \right) &= q \left( \chi^\beta_{w_1} + \chi^\beta_{w_2}, \chi^\delta_{V^+} \right) \tag{\EQnum} \label{eq:vw}\\ &\text{and} \notag \\
\stepcounter{claim}
  q \left( (\chi^\beta_{w_1} + \chi^\beta_{w_2})', \aI{x} \right) &= q \left( (\chi^\beta_{w_1} + \chi^\beta_{w_2})', c \right) \tag{\EQnum} \label{eq:vw'}
\end{align}(where \(+\) is boolean join and \('\) is boolean complement) iff \( \aI{x} = d_{|\{w_1^+,w_2^+\}} \oplus c_{|\mathrm{else}} \).

The direction (\(\Leftarrow\)) is a straightforward computation. For the forward direction, \[
i \in \{ w_1^+, w_2^+\} \quad \Longrightarrow \quad x^i = q(1,x^i) = q(1,(\chi^\delta_{V^+})^i) = d \]from equation (\ref{eq:vw}), and similarly \[
i \in \{ w_1^-, w_2^- \} \quad \Longrightarrow \quad x^i = q(1,x^i) = q(1,(\chi^\delta_{V^+})^i) = c \]while equation (\ref{eq:vw'}) yields \[
i \notin \{w_1^\pm, w_2^\pm \} \quad \Longrightarrow \quad x^i = p(1,x^i) = p(1,c) = c\]The proof of the claim is then accomplished by existentially quantifying \(\chi^\beta_{w_1},\chi^\beta_{w_2}\).

\begin{claim}
  If \( \aI{x} = d_{|\{w_1^+,w_2^+\}} \oplus c_{|\mathrm{else}} \in D \) then \( w_1 \Edge w_2 \).
\end{claim}

To see this, let \( \aI{x} = d_{|\{w_1^+,w_2^+\}} \oplus c_{|\mathrm{else}} = t( \vec{\chi}^\beta, \vec{\chi}^\delta) \in D \) for some polynomial \(t \in \Pol{\card{V} + \card{E} + 1}{\alg{A}}\). Without loss of generality, \(t\)'s image is contained in \(U\). By inspecting the \(v^-\) coordinates, we see that for any \(v \in V\) \[ t(0,\ldots,0,1,0,\ldots, c, \ldots, c) = c \](the \(1\) occuring in the \(v^{\mathrm{th}}\) place). Fix any \(w \in V\); then \[ x^{v^-} = t\left( (\vec{\chi}^\beta )^{w^-}, ( \vec{\chi}^\delta )^{v^-} \right)\]Moreover, since \( (\vec{\chi}^\beta)^{v^+} = (\vec{\chi}^\beta)^{v^-} \) for all \(v\) and \( \TC{\delta}{\{0,1\}}{\bot} \), one has \begin{align*}
  t\left( (\vec{\chi}^\beta)^{v^+}, (\vec{\chi}^\delta)^{v^-} \right) = t\left( (\vec{\chi}^\beta)^{v^-} \right. &,\left. (\vec{\chi}^\delta)^{v^-} \right) = t\left( (\vec{\chi}^\beta)^{w^-}, (\vec{\chi}^\delta)^{v^-} \right) \\ &\Downarrow \\
  t\left( (\vec{\chi}^\beta)^{v^+}, (\vec{\chi}^\delta)^{v^+} \right) &= t\left( (\vec{\chi}^\beta)^{w^-}, (\vec{\chi}^\delta)^{v^+} \right)
\end{align*}In other words, \[ x^i = t \left( (\vec{\chi}^\beta)^{w^-}, (\vec{\chi}^\delta)^i \right) \]for all \(i \in I\).
	
But as a polynomial on \(U\), \( t( (\vec{\chi}^\beta)^{w^-},\vec{y} ) \) depends only on one variable, say \(t((\vec{\chi}^\beta)^{w^-},\vec{y}) = f( y_k )\), with \(y_k\) corresponding to a generator \( \chi^\delta_k \in \vec{\chi}^\delta \); \(k\) is either an edge of \( \mathbb{G} \) or \( V^+ \). Since \(c,d\) are taken from the same trace and \(f\) does not collapse traces to points, we must have that \( \aI{x} = f(\chi^\delta_k) \) and \( x^i = x^j \) iff \( (\chi^\delta_k)^i = (\chi^\delta_e)^j \) for all \(i,j \in I \); since \( \card{V} > 2 \) and \(\aI{x} \) has \(d\) at only has two coordinates (out of at least six), \( k \) must be the edge  \(\langle w_1,w_2 \rangle \). This proves the Claim.
	
We can now complete the semantic interpretation: \( V \) is defined as the atoms of \( \{0,1\}^I \cap D \), and \( v \Edge w \) iff there exists \( \aI{x} \) as in Claim \ref{claim:edges (disjoint centralizer)} such that \( \chi^\beta_v \) and \( \chi^\beta_w \) are the two atoms witnessing the truth of the formula in that Claim.
\end{proof}

We are ready for the main result of this section:
\setcounter{claim}{0}
\renewcommand{\theclaim}{\ref{MainThm:sigma comparable}.\arabic{claim}}

\begin{proof}[Proof of Theorem \ref{MainThm:sigma comparable}]
  Let \( \alg{S} \) be subdirectly irreducible, with unary-type monolith \(\mu\); let \( \beta \) be incomparable to the strongly solvable radical \( \sigma \). Without loss of generality (see Fact \ref{fact:basics}), \( \typset{\alg{S}} = \{ 1,3 \}\), and some lower cover of \(  \beta \) is (strictly) below \( \sigma \). Choose \( \beta \land \sigma \prectype{1} \alpha \leq \sigma \); clearly \(  \beta \land \sigma = \alpha \land \beta =: \alpha \beta \prectype{3}  \beta \).
		
  Choose an \( ( \alpha \beta,  \beta ) \)-minimal set, which we may take without loss of generality to be polynomially isomorphic to the two-element boolean algebra \( \{0,1\} \); similarly without loss of generality, \(  \beta = \Cg{}{0}{1} \); also choose a \( (\alpha \beta, \alpha) \)-minimal set \(U\) containing elements \( c \equiv_{\alpha \setminus \alpha \beta} d \).
		
  Now, by Lemma \ref{lemma:zeta solvable}, the centralizer of \( \mu \) is solvable; hence \( \neg \TC{ \{0,1\}}{\mu}{\bot} \). By Lemma \ref{lemma:disjoint centralizer}, we may assume that the centralizer of \( \{0,1\} \) is the trivial congruence: for any \( a_1 \neq a_2 \) in \(\alg{S} \), there exists a polynomial \( t(x,\vec{y}) \) and tuples \( \vec{b}_0, \vec{b}_1 \) from \( \{0,1\} \) so that \begin{align*}
    t(a_1,\vec{b}_0) &=    t(a_1, \vec{b}_1) \\ &\text{but} \\
    t(a_2,\vec{b}_0) &\neq t(a_2, \vec{b}_1)
  \end{align*}Since \( \alg{S}_{|\{0,1\}} \) is a boolean algebra, the discussion after equation \eqref{eq:TCsimplification} shows that we can transform this package into one using a binary polynomial:\begin{align*}
    s(a_1,0) &= s(a_1,1) \\ &\text{but} \\
    s(a_2,0) &\neq s(a_2,1)
  \end{align*}witnessing that \( \{ a_1, a_2 \} \) does not centralize \( \{0,1\} \).
	
  Our strategy is to interpret the class of graphs with at least five vertices into \( \HSP{\alg{S}} \), so let \( \mathbb{G} = \langle I, E \rangle \) be any graph. Define \( \alg{D} = \alg{D}(\mathbb{G}) \leq \alg{S}^I \) to be the subalgebra generated by the constants together with all points \[ \chi^\beta_i := 1_{|i} \oplus 0_{|\mathrm{else}} \qquad (i \in I)\]and \[ \chi^\alpha_e := d_{|\{i,j\}} \oplus c_{|\mathrm{else}} \qquad (e = \{i,j\} \in E) \]By the usual arguments, \( \{0,1\}^I \subseteq D \) is a definable subset, as is the set of its atoms.
		
  Let \( \chi^\beta_i \) be any atom in \( \{0,1\}^I \). Let \( \aI{y}, \aI{z} \) be any elements of \( \alg{D} \). Then \begin{align*}
    p(0, \aI{y}) &= p(\chi^\beta_i,\aI{y}) \\ &\Updownarrow \\
    p(0, \aI{z}) &= p(\chi^\beta_i,\aI{z})
  \end{align*}for all \( p \in \Pol{2}{\alg{S}} \) if and only if \(y^i\) and \(z^i\) are congruent modulo the centralizer of \( \{0,1\} \), i.e. are equal. But \( \alg{S} \) only has finitely many binary polynomial operations; hence the above condition is a first-order property \( \Phi(\chi^\beta_i, \aI{y}, \aI{z}) \): we have proved
	
  \begin{claim}
    If \( s \in S \), \(i \in I \), \( \aI{y} \in D \) then \(y^i = s \) iff \( \Phi(\chi^\beta_i,\aI{y},s) \).
  \end{claim}Or in plainer English: \( \alg{D} \) knows its own product structure.

  In particular: the set of those \( \aI{x} \in U^I \cap D\) of the form \( d_{|\{i_0,i_1\}} \oplus c_{|\mathrm{else}} \) for precisely two vertices \( i_0, i_1 \), is a definable subset. The generators \( \chi^\alpha_e \) belong to this set; we will be done if we can show that

  \begin{claim}
    If \(i_0 \neq i_1 \) and \(\aI{x} = d_{|\{i_0,i_1\}} \oplus c_{|\mathrm{else}} \in D\) then \( i_0 \Edge i_1 \).
  \end{claim}

  So let \[ \aI{x} = d_{|\{i_0,i_1\}} \oplus c_{|\mathrm{else}} = f(\vec{\chi}^\beta, \vec{\chi}^\alpha)\]belong to \(D\), where \( f \in \Pol{\card{I} + \card{E}}{\alg{S}} \) takes values in \(U\) and \( \vec{\chi}^\alpha \), \( \vec{\chi}^\beta \) enumerate the two sets of generators.

  Let \(j \in I\) be any vertex. Then \[ x^j = f \left( (\vec{\chi}^\beta)^j, (\vec{\chi}^\alpha)^j \right) \equiv_\beta f \left( (\vec{\chi}^\beta)^{i_0}, (\vec{\chi}^\alpha)^j \right) \equiv_\alpha f \left( (\vec{\chi}^\beta)^{i_0}, (\vec{\chi}^\alpha)^{i_0} \right) = x^{i_0} \equiv_\alpha x^j \]Hence \[ \aI{x} \equiv_{\alpha \beta} f \left( (\vec{\chi}^\beta)^{i_0}, \vec{\chi}^\alpha \right) \]But considered as a mapping from \(\alpha\)-classes into \(U\), \(f((\vec{\chi}^\beta)^{i_0}, \vec{v})\) depends modulo \( \alpha \beta \) on no more than one of the edge-variables, say \(f((\vec{\chi}^\beta)^{i_0},\vec{v}) = g(v_e) \) for some \(e  = \{j_0,j_1\} \in E\); since \( \aI{x} \) is not constant modulo \( \alpha \beta \), \( g \) cannot collapse traces to points, implying that \( \aI{x} = g(\chi^\beta_e)\) has the same \( \alpha \beta \)-equivalence pattern as \( \chi^\alpha_e \). The two equal coordinates of \( \aI{x} \) must match two equal coordinates of \( \chi^\alpha_e \) such that all other coordinates have a different value; since \( \card{V} > 4 \), the only set of such coordinates is \( \{j_0, j_1\} \);  but this implies \( \aI{x} = \chi^\alpha_e \), as desired.
\end{proof}
\renewcommand{\theclaim}{\arabic{section}.\arabic{thm}.\arabic{claim}}

The investigations of \emph{congruence modular} finitely decidable varieties identified quite early how constrained the congruence geometry of such varieties must be. In particular, it was discovered that the congruences above the solvable radical of a subdirectly irreducible algebra in such a variety were forced to be linearly ordered. Theorem \ref{MainThm:sigma comparable} allows us to remove the hypothesis of modularity:

\begin{cor}\label{cor:sigma meet-irreducible}
  Let \( \alg{S} \) be a finite subdirectly irreducible algebra with unary-type monolith. If the congruence interval above the solvable radical of \( \alg{S} \) is not linearly ordered, then \( \HSP{\alg{S}} \) is hereditarily finitely undecidable.
\end{cor}

\begin{proof}
  Due to the transfer prinicples (see Fact \ref{fact:basics}), we already know that \( \typset{\alg{S}} \subset \{ 1,3 \}\); without loss of generality, the solvable radical \( \mathrm{Rad}(\alg{S}) < \top_S \) and every cover above \( \Rad{\alg{S}} \) has boolean type. If \( \Rad{\alg{S}} \) were to have just one upper cover, then \( \alg{S}/\Rad{\alg{S}} \) would be subdirectly irreducible with boolean monolith; Idziak's characterization (\cite{Idziak97}) implies then the whole interval \( [\mathrm{Rad}(\alg{S}),\top] \) would be a chain. Hence it suffices to show that the radical having at least two upper covers \( \alpha_0, \alpha_1 \) leads to a contradiction.

  Theorem \ref{MainThm:sigma comparable} implies that every subcover of \( \Rad{\alg{S}} \) is meet-irreducible, so without loss of generality (by passing to a quotient by such a subcover) we may assume that \( \bot \stackrel{1}{\prec} \Rad{\alg{S}} =: \mu \).  Let \( K_a = \{ 0_a, 1_a \} \) be respectively \( (\mu, \alpha_a) \)-minimal sets (\(a \in \{0,1\} \)).

  We know that \( \neg \TC{K_a}{\mu}{\bot} \) for \(a = 0,1\), since each of these sets generate a congruence above the centralizer of \(\mu\). By Lemma \ref{lemma:disjoint centralizer}, we may also assume that \( \neg \TC{\mu}{K_a}{\bot} \). Let \begin{align*}
    p_0(c,0_0) &= p_0(c,1_0) \\ &\text{but} \\
    p_0(d,0_0) &\neq p_0(d,1_0)
\end{align*}witness this latter failure. Observe that \( p_0(d,0_0) \equiv_\mu p_0(d,1_0) \); hence there exists \( q \in \Pol{1}{\alg{S}} \) taking \( K_0 \) injectively into some \( (\bot, \mu)\)-trace \(N\). Since \( \mu = \Cg{\alg{S}}{q(0_0)}{q(1_0)} \), we must have \( \neg \TC{\{ q(0_0), q(1_0)\}}{K_1}{\bot} \). Choose a witnessing package \begin{align*}
    p_1(q(0_0),0_1) &= p_1(q(0_0),1_1) \\ &\text{but} \\
    p_1(q(1_0),0_1) &\neq p_1(q(1_0),1_1)
\end{align*}

  Our strategy is to interpret the undecidable class \( \EQtwo \) (see page \pageref{defn:two equivalence relations}) into the diagonal subpowers of \( \alg{S} \). So let \( \alg{E} = \langle I; R_0, R_1 \rangle \models \EQtwo \), and define a diagonal subpower \( \alg{D} = \alg{D}(\alg{E}) \leq \alg{S}^I \) as the subalgebra consisting of all \( \aI{x} \in \alg{S}^I \) such that \( \aI{x} \) is \( \alpha_0 \)-constant on each block of \( R_1 \) and \( \alpha_1 \)-constant on each block of \(R_0\). Note that, since \(0_0, 1_0\) are \(\alpha_0\)-congruent but not \( \alpha_1\), a point \( \aI{x} \in K_0^I \) belongs to \( \alg{D} \) iff it is constant on each \( R_0 \)-block. We conclude that \( \alg{D}_{|K_0} \) is canonically isomorphic to the boolean algebra \( \alg{2}^{I/R_0} \); the corresponding facts hold mutatis mutandis for \(K_1^I\). Furthermore, these two subsets are uniformly definable (by Lemma \ref{prop:power of minimal set}). Let \( \mathrm{AT}_a(v) \) be a formula asserting that \(v\) is an atom of the boolean algebra \(\alg{D}_{|K_a}\), and let \(H\) be the (definable) set of pairs \( \langle \aI{y}, \aI{z} \rangle \) such that \( \aI{y} = {1_0}_{|B_y} \oplus {0_0}_{|\mathrm{else}} \) is a \(K_0\)-atom and \( \aI{z} = {1_1}_{|B_z} \oplus {0_1}_{|\mathrm{else}} \) is a \(K_1\)-atom.

  Now, for each pair \( \langle \aI{y}, \aI{z} \rangle \in H\), the blocks coded by the two points are either empty or share one \( i \in I \). Write \( \aI{y} \bowtie \aI{z} \) if the intersection is nonempty. It suffices to show that the relation \( \aI{y} \bowtie \aI{z} \) is definable. Why is this so? Since \( R_0 \cap R_1 \) is trivial, every \( i \in I \) corresponds canonically to exactly one \( \langle \aI{y}_i, \aI{z}_i \rangle \in H\), namely \( \aI{y}_i = {(1_0)}_{|i / R_0} \oplus {(0_0)}_{|\mathrm{else}}\) and \( \aI{z}_i = {(1_1)}_{|i / R_1} \oplus {(0_1)}_{|\mathrm{else}}\). These two points are \(\bowtie\)-related by construction. But if \(\bowtie\) is definable, the structure \( \alg{E} \) can be recovered on the underlying set \( \bowtie \; = \{ \langle \aI{y}_i, \aI{z}_i \rangle \colon i \in I \} \) using the first-order theory of \( \alg{D} \), since \( \langle i, j \rangle \in R_0 \) (resp \(R_1 \)) iff \( \aI{y}_i = \aI{y}_j \) (resp \( \aI{z}_i = \aI{z}_j \)).

  To this end, observe: if \( \langle \aI{y}, \aI{z} \rangle \in H \) and \( i \in I \), \[ p_1(q(y^i),z^i) \neq p_1(q(y^i),0_1) \iff z^i \neq 0_1 \text{ and } y^i \neq 0_0 \]It follows that\begin{align*}
    p_1(q(\aI{y}), \aI{z}) &\neq p_1(q(\aI{y}),0_1) \\ &\Updownarrow \\
    p_1(q(y^i),z^i) &\neq p_1(q(y^i),0_1) \text{ for some } i \in I \\ &\Updownarrow \\
    y^i = 1_0 &\text{ and } z^i = 1_1 \text{ for some } i \in I \\ &\Updownarrow \\
    \aI{y} &\bowtie \aI{z} \hskip 1cm \qedhere
  \end{align*}
\end{proof}

\section{The strongly solvable radical is strongly abelian}\label{sec:Abelian}
\setcounter{thm}{0}

In this section, we prove Theorem \ref{MainThm:sigma strongly abelian}. The proof proceeds through three increasingly complex semantic interpretation constructions, all of a highly tame-congruence-theoretic nature.

\begin{lemma}\label{lemma:sigma centralizes minimal sets}
  Let \(\alg{S}\) be a subdirectly irreducible algebra with unary-type monolith \(\mu\) and strongly solvable radical \(\sigma\) which is abelian over \(\mu\) but not over \(\bot_S\). Let \(U = e(S)\) be any \((\bot_S,\mu)\)-minimal set. If \(\TC{\sigma}{\mu_{|U}}{\bot}\) fails in \(\alg{S}\), then \(\HSP{\alg{S}}\) is hereditarily finitely undecidable.
\end{lemma}

\begin{proof}
  Since \(\TC{\mu}{\mu_{|U}}{\bot}\) always holds, we may climb the congruence lattice until we get a cover \(\mu \leq \theta_0 \stackrel{1}{\prec} \theta_1 \leq \sigma\) such that \(\TC{\theta_0}{\mu_{|U}}{\bot}\) holds and \(\TC{\theta_1}{\mu_{|U}}{\bot}\) does not. Fix a \((\theta_0,\theta_1)\)-minimal set \(U' = e'(S)\) with trace \(N'\) containing \(\theta_0\)-inequivalent elements \(a_0,a_1\). Since these elements generate \(\theta_1\) over \(\theta_0\), already \(\neg \TC{\Cg{}{a_0}{a_1}}{\mu_{|U}}{\bot}\), and we may take a witnessing package \begin{align*}
    t(a_0,\vec{b}_0) &= t(a_0,\vec{b}_1) \\ &\text{but} \\
    t(a_1,\vec{b}_0) &\neq t(a_1,\vec{b}_1)
  \end{align*}There is no loss of generality in assuming that the image of \(t\) is contained in \(U\).

  Since \(\mu\) is strongly abelian, we may assume that \(\vec{b}_0\) and \(\vec{b}_1\) differ only in one place (say the first), so that for \(q(v_1,v_2) = t(v_1,v_2,b^2,\ldots,)\), the polynomial \(q(a_0,x)\) is constant on \(\mu_{|U}\)-blocks while the polynomial \(q(a_1,x)\) permutes \(U\). (Observe that \(q(x,y) \in U\) for any \(x,y \in S\).) Of course we may by iterating \(q\) guarantee that for each \(u' \in U'\), the operation \(q(u',x)\) is idempotent. The same argument shows that for each \(u' \in U'\), \(q(u',x)\) is either the identity on \(U\) (in which case we call \(u'\) \emph{permutational}) or else squashes each \(\mu\)-block of \(U\) to a point (at which we call \(u'\) \emph{collapsing}). Since \(\TC{\theta_0}{\mu_{|U}}{\bot}\), these two properties are invariant under \(\theta_0\)-congruence.

  Let \(N \subseteq U\) be any trace; we have that \(q(a_0,N) = m_0\) for some \(m_0 \in U\). In fact, since \(\sigma\) is abelian over \(\mu\), we have \begin{align*}
    q(a_0,m_0) &= q(a_1,m_0) \\ &\Downarrow \\
    m_0 = q(a_0,u) &\equiv_\mu q(a_1,u) = u \qquad \text{ for any } u \in N
  \end{align*}and thus \(m_0 \in N\); more generally, we have that the polynomial \(v_1 \mapsto q(a_0,v_1) \) retracts each trace down to one of its points. Since \(N\) was a trace, there exists some \(m_1 \neq m_0\) in \(N\), which we fix for future use.

  We want to semantically embed graphs into the diagonal subpowers of \(\alg{S}\), so let \(\mathbb{G} = \langle V,E \rangle\) be a graph. Our index set \(I\) will equal \(V \sqcup \{ \infty \}\). Our subpower \(\alg{S}[\mathbb{G}]\) will be the subalgebra of \(\alg{S}^I\) generated by the diagonal together with \begin{itemize}
  \item for each vertex \(v \in V\), the element \[\aI{g}_v = {a_1}_{|\{v,\infty\}} \oplus {a_0}_{|\mathrm{else}}\]
  \item for each edge \(\{v_1,v_2 \} \in E\), the element \[ \aI{g}_{v_1 v_2} = {a_1}_{|\{v_1,v_2,\infty\}} \oplus {a_0}_{|\mathrm{else}}\]and
  \item the element \[\chi_\infty = {m_0}_{|V} \oplus {m_1}_{|\infty}\]
  \end{itemize}

  Recall our notational convention (page \pageref{defn:diagonal subpower}) that for \(s \in S\) we will use a boldface \(\aI{s}\) to denote the corresponding diagonal element; let \( \vec{\aI{s}} \) be a fixed enumeration of these diagonal elements. Observe that each generator, and hence every element of \(\alg{S}[\mathbb{G}]\), is constant modulo \(\theta_1\); and that \(\chi_\infty\) is also constant mod \(\theta_0\) (indeed, mod \(\mu\)). 
	
  \begin{claim}\label{claim:two_theta_values}
    Every element of \((U')^I \cap \alg{S}[\mathbb{G}]\) assumes at most two values\(\pmod {\theta_0}\), with one supported either on all of \(I\), or on \(\{v,\infty \}\) (for some \(v \in V\)), or on \(\{v_1,v_2,\infty\}\) (for some \(v_1 \Edge v_2\)).
  \end{claim}

  (As on page \pageref{defn:diagonal subpower}, we will drop the ``\( {} \cap \alg{S}[\mathbb{G}] \)'' when the context is unambiguous.)

  Let \(\aI{x} = t(\aI{g}_v, \ldots, \aI{g}_{v_1 v_2},\ldots, \chi_\infty, \vec{\aI{s}})\) represent an arbitrary element of \(\alg{S}[\mathbb{G}]\) all of whose coordinates lie in \(U'\). Without loss of generality (by precomposing with \(e'\)) \(t\) respects \(U'\); but then this operation is sensitive (mod \(\theta_0\)) to changes (mod \(\theta_1\)) in no more than one of its variables. Since all generators are constant (mod \(\theta_1\)), we conclude that the blocks of \(I\) on which \(\aI{x}\) is constant (mod \(\theta_0 \)) coincide with those of whichever generator sits at the active place. The claim follows immediately.

  We now identify a subset \(\Gamma\) of the universe, definable (using parameters for the diagonal elements and \(\chi_\infty\)) and a definable preorder \(\ll\) on \(\Gamma\).

  Set \[ \Gamma = \left\{ \aI{x} \in (U')^I : q(\aI{x},\aI{m}_0) = \aI{m}_0 \; \& \; q(\aI{x},\chi_\infty) = \chi_\infty \right\} \]and preorder it by \[ \aI{x} \ll \aI{y} \iff \forall \aI{u}, \aI{v} \in U^I \; q(\aI{x},\aI{u}) = q(\aI{x},\aI{v}) \rightarrow q(\aI{y},\aI{u}) = q(\aI{y},\aI{v}) \]Since the sets \(U^I\) and \((U')^I\) are definable (Proposition \ref{prop:power of minimal set}), it follws that \( \ll \) and its associated equivalence relation \( \sim \) are definable too. Let \(\mathrm{EQ}(v_1,v_2)\) be a formula defining the equivalence \(\sim\).

  The second conjunct defining \(\Gamma\) implies that if \(\aI{x} \in \Gamma\) then \(\aI{x}\) is permutational at infinity. (So, for example, \(\Gamma\) contains \(\aI{a}_1\) but not \( \aI{a}_0 \).) The first implies that any non-permutational factor of \(\aI{x}\) must collapse \(N\) to \(m_0\).	If \(\aI{x} \in \Gamma\), \(\aI{u}_1,\aI{u}_2 \in U^I\), and \(x^i\) is not permutational, then \(q(\aI{x},\aI{u}_1) = q(\aI{x},\aI{u}_2)\) implies \(u_1^i \equiv_\mu u_2^i\).

  \begin{claim}
    For \(\aI{x} \in \Gamma\), define \begin{align*} \mathrm{supp}(\aI{x}) &= \{i \in I : x^i \text{ is permutational} \} \\ &= \{i \in I : q(x^i,m_1) = m_1 \} \end{align*} (We already know that each support is either \(I\) or one of the sets \(\{v_1,v_2,\infty \}\) (\(v_1 \Edge v_2\)) or \(\{v,\infty\}\) (\(v \in G\)).) Then \[\aI{x} \ll \aI{y} \iff \mathrm{supp}(\aI{x}) \supseteq \mathrm{supp}(\aI{y}) \]
  \end{claim}

  (\(\Rightarrow\)): If \(v \in \mathrm{supp}(\aI{y}) \setminus \mathrm{supp}(\aI{x})\), take \(\aI{u} = q(\aI{g}_v,\aI{m}_1)\). Then \begin{align*} q(\aI{x},\aI{u}) = \chi_\infty &= q(\aI{x},\chi_\infty) \\ &\text{but} \\ q(\aI{y},\aI{u})_{|v} = q(y^v,m_1) = m_1 &\neq m_0 = q(\aI{y},\chi_\infty)_{|v} \end{align*}so \(\aI{x} \not \ll \aI{y}\).
	
  (\(\Leftarrow\)): For \(\aI{t},\aI{u} \in U^I\), \(q(\aI{x},\aI{t}) = q(\aI{x},\aI{u})\) is equivalent to \[ \aI{t}_{|\mathrm{supp}(\aI{x})} = \aI{u}_{|\mathrm{supp}(\aI{x})} \text{ and for } v \notin \mathrm{supp}(\aI{x}),\: t^v \equiv_\mu u^v\]which implies \[ \aI{t}_{|\mathrm{supp}(\aI{y})} = \aI{u}_{|\mathrm{supp}(\aI{y})} \text{ and for } v \notin \mathrm{supp}(\aI{y}),\: t^v \equiv_\mu u^v\]which is equivalent to \(q(\aI{y},\aI{t}) = q(\aI{y},\aI{u})\).
  
  As an immediate consequence, we have that every \(\aI{x} \in \Gamma\) is \(\sim\) to exactly one of \(\{ \aI{a}_1 \} \cup \{ \aI{g}_{v_1 v_2} : v_1 \Edge v_2 \} \cup \{ \aI{g}_v : v \in V \} \). The quotient partial order on \(\Gamma / \sim\) has height two, with \(\aI{a}_1\) at level zero, all the edges at level one and all the vertices at level two. 
	
  Let \(\mathrm{WHO}(v_1)\) be a formula asserting that \(v_1 \in \Gamma\) and \(v_1\) is at \(\ll\)-level two. We have just observed that the map \(w \mapsto \aI{g}_w / \sim\) is a bijection of \(V\) with the extension of \(\mathrm{WHO}(v_1)\) modulo \(\sim\) (which was already found to be a definable equivalence relation). Let \(\mathrm{EDGE}(v_1,v_2)\) be a formula asserting that there exists \(\aI{y} \in \Gamma\) at \(\ll\)-level one such that \(\aI{y} \ll v_1 \; \& \; \aI{y} \ll v_2\). Then these formulas recover the structure of \(\mathbb{G}\).
\end{proof}

The conclusions of the following lemma can be shown to hold for either of the solvable radical, or the strongly solvable radical, of any finite algebra \(\alg{A} \); however, the proof of this more general theorem is no more enlightening for our purposes, so we omit it.

\begin{lemma}\label{lemma:sigma_definable}
  If \(\alg{A}\) is any finite algebra in a finitely decidable variety with strongly solvable radical \(\sigma\), there exists a first-order formula with parameters from \(A\) which defines the congruence \(\sigma^I / \Theta\), uniformly for all \(\alg{D}/\Theta\), where \(I\) is any index set, \(\Delta \leq \alg{D} \leq \alg{A}^I\) is any diagonal subpower, and \(\Theta \leq \sigma^I \cap \alg{D} \in \Con{\alg{D}}\).
\end{lemma}

\begin{proof}
  The argument comes from the theory of snags (see \cite{Hobby_McKenzie} Chapter 7). Let \(E(\alg{A})\) denote the collection of all idempotent polynomials with nontrivial range, and for each \(e \in E(\alg{A})\) choose \(p \in \Pol{3}{\alg{A}}\) which is Malcev on the image of \(e\) if any such polynomial exists; if none, then let \(p\) be second projection. Then we have that a pair \(\langle x, y \rangle\) fails to belong to \(\sigma\) iff there is a congruence cover \(\alpha \stackrel{2,3}{\prec} \beta\) below \(\Cg{}{x}{y}\) iff the following first-order formula is satisfied: \begin{align*}\bigvee_{e \in E(\alg{A})} \bigvee_{f \in \Pol{1}{\alg{A}}} & ef(y) = p(ef(y),ef(x),ef(x)) = p(ef(x),ef(x),ef(y)) \\ &\quad \neq p(ef(x),ef(x),ef(x)) = ef(x) \end{align*}The formula is clearly false if every cover below \(\langle x,y \rangle\) has type 1, while a cover of boolean or affine type will guarantee the formula's truth, since the minimal sets of that cover have empty tails and hence Malcev polynomials. This proves that the indicated formula defines \(\sigma\) in \(\alg{A}^1\), and its truth is preserved by factoring out by congruences under \(\sigma\).
	
  Now since the defining formula is quantifier-free, it is preseved in subpowers. Finally, if \(\aI{x} \equiv_{\sigma^I} \aI{y}\), \(e \in E(\alg{A})\), \(p(v_1,v_2,v_3) = v_2\) and \(f \in \Pol{1}{\alg{A}}\), \[p(ef(\aI{y}),ef(\aI{x}),ef(\aI{x})) = p(ef(\aI{x}),ef(\aI{x}),ef(\aI{y})) = p(ef(\aI{x}),ef(\aI{x}),ef(\aI{x})) = ef(\aI{x}) \]which is preserved under factoring out \(\Theta\). On the other hand, if \(x^i \not \equiv_\sigma y^i\), then the polynomials which witness \[ef(y^i) \equiv_\theta p(ef(y^i),ef(x^i),ef(x^i)) \not \equiv_\theta ef(x^i) \](\(\theta\) being the projection of \(\Theta\) into the \(i\)th coordinate) also witness it in \(\alg{D}\).
\end{proof}

\begin{defn} Let \(\alg{A} \) be any algebra, \(U \subseteq A\), and \(\sigma\) be the strongly solvable radical of \(\alg{A}\). We write
  \begin{enumerate}
  \item \( \PolGrp{\alg{A}}{U} := \Pol{1}{\alg{A}_{|U}} \cap \mathfrak{S}(U) \) for the group of permutations of \(U\) realized as polynomials of \(\alg{A}\), and
  \item \( \TwinGrp{\alg{A}}{U} \) for the subgroup consisting of those \(f \in \PolGrp{\alg{A}}{U} \) such that for some term \(t(v_0, \ldots, v_n) \) and some \( \vec{d} \equiv_\sigma \vec{e} \) we have \[\alg{A}_{|U} \models v_0 = t(v_0, \vec{e}) \; \& \; f(v_0) = t(v_0, \vec{d}) \](Such a permutation is known as a \( \sigma \)-twin of the identity.)
  \end{enumerate}
\end{defn}

A straightforward computation shows that \( \TwinGrp{\alg{A}}{U} \) is normal in \( \PolGrp{\alg{A}}{U} \).

Note that there is nothing special about the solvable radical in this context; we can define \( \alpha \)-twins for any congruence \( \alpha \), but since we will be exclusively concerned with \( \sigma \)-twins in this investigation, we will leave the definition more specialized so as to avoid needing a third parameter in the symbol \(\TwinGrp{\alg{A}}{U} \).

\begin{prop}\label{prop:Pi acts transitively}
  Let \( \alg{A} \) be a finite algebra. If \(\bot_A \stackrel{1}{\prec} \mu \) in \( \Con{\alg{A}} \) and \(U\) is \( (\bot,\mu) \)-minimal, then
  \begin{enumerate}
    \item \( \PolGrp{\alg{A}}{U} \) acts transitively by polynomial isomorphisms on the set of traces inside \(U\);
    \item the action of \( \PolGrp{\alg{A}}{U} \) on the body of \(U\) has at most two orbits; 
    \item if some \(f \in \PolGrp{\alg{A}}{U} \) nontrivially permutes some trace, then \( \PolGrp{\alg{A}}{U} \) acts transitively on the body of \(U\).
  \end{enumerate}
\end{prop}

\begin{proof}
  That \( \PolGrp{\alg{A}}{U} \) acts on traces is an easy consequence of the fact that \( \mu \) is a congruence of the algebra.

  To transitivity: \(\mu\) is generated by any of its nontrivial pairs, so let \(N_i \subseteq U\) (\(i = 1,2\)) be traces containing elements \(a_i \neq b_i\). Then we can string a chain of elements \[ a_2 = u_0 \neq u_1 \neq \cdots \neq u_{m+1} = b_2\]where \( \{u_j, u_{j+1} \} = \{f_j(a_1),f_j(b_1)\} \) for some \(f_j \in \PolGrp{\alg{A}}{U} \). Then \(f_m(N_1) = N_2 \). This argument actually shows that \(b_2 \in \PolGrp{\alg{A}}{U}(a_1) \cup \PolGrp{\alg{A}}{U}(b_1)\), which proves the second and third statements.
\end{proof}

\begin{lemma}\label{lemma:twins act trivially}
  Let \(\alg{S}\) be a finite subdirectly irreducible algebra with type-1 monolith \(\mu\) and strongly solvable radical \( \sigma \) satisfying \( \TC{\sigma}{\sigma}{\mu} \). Let \(U = e(S)\) be a \( (\bot,\mu) \)-minimal set. If \( \TwinGrp{\alg{S}}{U} \) nontrivially permutes some trace, then \( \HSP{\alg{S}} \) is hereditarily finitely undecidable.
\end{lemma}

\begin{proof}
  The last statement of Proposition \ref{prop:Pi acts transitively} ensures that \( \PolGrp{\alg{S}}{U} \) acts transitively on the body of \(U\); the same may not be true of the induced action of \( \TwinGrp{\alg{S}}{U} \), but elementary group theory shows that \( \PolGrp{\alg{S}}{U} / \TwinGrp{\alg{S}}{U} \)	acts in a well-defined and transitive way on the orbits of the action by \( \TwinGrp{\alg{S}}{U} \). Since the action of \( \PolGrp{\alg{S}}{U} \) is transitive, we will use the symbol \( \mathcal{O}(a) \) exclusively to refer to the orbit of the element \(a \in \mathrm{Body}(U) \) under the action by \( \TwinGrp{\alg{S}}{U} \).

  \begin{claim}
    For each \(c \in \mathrm{Body}\), \[ | \mathcal{O}(c) \cap N | > 1 \]where \(N\) is the trace containing \(c\).
  \end{claim}
	
  Let \(g(a) = b \equiv_{\mu \setminus \bot} a\) and \(f(c) = a\), where \( g \in \TwinGrp{\alg{S}}{U} \) is the hypothesized nontrivial permutation of a trace and \( f \in \PolGrp{\alg{S}}{U} \). Then \( f^{-1} \circ g \circ f(c) \equiv_{\mu \setminus \bot} c \), which proves the claim.

  By Lemma \ref{lemma:sigma centralizes minimal sets}, we may assume that \( \TC{\sigma}{\mu_{|U}}{\bot} \). This immediately implies that if \( t(v_0, \ldots, v_n) \) is any term and \( \vec{c} \equiv_\sigma \vec{d} \), and if \(t(U,\vec{c}), t(U,\vec{d}) \subseteq U \) then these two polynomials are either both permutations of \(U\) or both collapse traces into points.

  Our plan is a bit more complicated this time around. Instead of semantically embedding graphs into diagonal subpowers of \(\alg{S}\), we will embed them into algebras \( \alg{C}[\mathbb{G}] = \alg{D}(\mathbb{G})/\Theta \), where \( \alg{D}(\mathbb{G}) \leq \alg{S}^I \) is a diagonal subpower of \( \alg{S} \) and \( \Theta \leq \sigma^I \). We will not attempt to show that \( \Theta \) is a definable congruence, uniformly or otherwise.

  Fix your favorite graph \( \mathbb{G} = \langle V, E \rangle \). Define \( V^\pm = \{ v^+,v^- \colon v \in V \} \) (the disjoint union of two copies of \(V\)), and set \( I = V^\pm \sqcup \{ \infty \} \); each of the sets \( \{ v^+, v^- \}\) as well as \( \{ \infty \} \) will be called a ``vertex block'' or ``\(V\)-block''. Let \( \alg{D} = \alg{D}(\mathbb{G}) \leq \alg{S}^I\) be generated by the set \( \bar{\Gamma} \) which is the disjoint union of the following three sets: \begin{itemize}
  \item \( \Gamma_0 \) is the set of those \( \aI{x} \in U^I \) which are constant on each \(V\)-block and constant\(\pmod \sigma \) on all of \(I\).
  \item \( \Gamma_V \) is the set of those \( \aI{x} \in \alg{S}^I \) such that for some \(a \in \mathrm{Body} \), \( x^i \in (a/\sigma) \cap \mathcal{O}(a) \) for all \(i \in I \), and for one \( v \in V \), \(x^{v^+} \equiv_{\mu \setminus \bot} x^{v^-}\), while for all \(w \neq v \), \(x^{w^+} = x^{w^-} \). For convenience, if \( \aI{x} \) and \(v\) are as just described, we write \( \mathrm{Label}(\aI{x}) = \langle v,x^{v^+} \rangle \).
  \item \( \Gamma_E \) is like \( \Gamma_V \); but instead of having one nonconstant vertex block, each point will have two, at the blocks of \(v\) and \(w\), where \( v \Edge w \), and write \( \mathrm{Label}(\aI{x}) = \langle v, x^{v^+}, w, x^{w^+} \rangle \).
  \end{itemize}

  We will refer to the non-constant vertex blocks as ``spikes''.

  Observe that since each generator is constant modulo \(\sigma\), every element of \(\alg{D}\) is too.
	
  \begin{claim}\label{claim:structure of T}
    \( \alg{D} \cap U^I \subset \bar{\Gamma}\), and for every polynomial \( \aI{f} \in \PolGrp{\alg{D}}{U} \) and every \( v \in V \), the \(v^+\) component of \(\aI{f}\) is the same function as the \(v^-\) component.
  \end{claim}

  (As is our convention, \( \PolGrp{\alg{D}}{U} \) should really more precisely be \( \PolGrp{\alg{D}}{U^I \cap D} \), but that would be cumbersome.)
	
  Both parts of the claim are consequences of Maroti's Lemma. To the first: let \( \aI{y} = et(\Gamma_0, \Gamma_V, \Gamma_E) \) be a typical element of \( \alg{D} \cap U^I \). There is one special input place where this term is sensitive to changes by \( \mu \); at all other places, flatten out all the spikes so that \( \aI{y} = et'(\Gamma_0,\aI{x}) \) where \( \aI{x} \in \Gamma \) is the element at the special place. Then if \( \aI{x} \in \Gamma_0 \) or if \(et'(\Gamma_0,\cdot) \) is not injective on \(U^I\) then at each coordinate \(t'({\Gamma_0}_{|i},\cdot) \) collapses \(\mu\) into points; under those hypotheses, \( \aI{y} \in \Gamma_0 \).
	
  On the other hand, if \( et'(\Gamma_0,\cdot) \) permutes \(U^I\), then \( \aI{y} \) has the same spike pattern that \( \aI{x} \) had (since every element of \(\Gamma_0\) is constant on V-blocks); furthermore, if \( \aI{x} \in \Gamma_V \cup \Gamma_E \), we can conclude that \( \aI{y} \) takes all its values from one \(\TwinGrp{\alg{S}}{U}\)-orbit, since all the coordinatewise polynomials \(et'({\Gamma_0}_{|i},\cdot) \) are \(\sigma\)-twins, and hence all in the same coset mod \( \TwinGrp{\alg{S}}{U} \). But this means that \( \aI{y} \in \bar{\Gamma} \) already.

  Similarly for the second part of the claim: let \( f(v_0) = et(v_0, \bar{\Gamma}) \in \PolGrp{\alg{D}}{U}\); then it is not possible for the special variable to be anything except the first. The claim follows, since the other parameters only vary up to \( \mu \) on vertex blocks.

  In fact, let \( T(v_0, \ldots, v_n) \) be a universal term for \( \TwinGrp{\alg{S}}{U} \), i.e. there exist pairwise-\(\sigma\) tuples \( \{\vec{d}_g \colon g \in \TwinGrp{\alg{S}}{U} \} \) so that \( g(v_0) = T(v_0, \vec{d}_g) \) for all \(g\). (We leave it to the reader to verify that such a term exists.) Then this term allows us to realize the full product \( {\TwinGrp{\alg{S}}{U}}^{V \sqcup \{ \infty \} } \) as polynomial permutations of \( U^I \); it follows that \( \PolGrp{\alg{D}}{U} \) is isomorphic to the inverse image of the diagonal subgroup under the canonical projection \[ {\PolGrp{\alg{S}}{U}}^{V \sqcup \{ \infty \}} \longrightarrow (\PolGrp{\alg{S}}{U} / \TwinGrp{\alg{S}}{U})^{V \sqcup \{ \infty \}} \]

  For the remainder of this proof, let \( \Gamma = U^I \cap \bar{\Gamma} = U^I \cap D \).

  We still have to define the congruence \(\Theta\). This is done as follows: \( \Theta \) will be generated by identifying those pairs \( \langle \aI{x}, \aI{y} \rangle \) such that \begin{itemize}
  \item \( \aI{x}, \aI{y} \in \Gamma_0\) and for all \( i \in I \), \( x^i \equiv_\mu y^i \), or
  \item \( \aI{x}, \aI{y} \in \Gamma_V\), \( \mathrm{Label}(\aI{x}) = \mathrm{Label}(\aI{y}) = \langle v,a \rangle \) and for all \(i \neq v^+ \), \( x^i \equiv_\mu y^i \), or
  \item \( \aI{x}, \aI{y} \in \Gamma_E\), \( \mathrm{Label}(\aI{x}) = \mathrm{Label}(\aI{y}) = \langle v,a,w,b \rangle \), and for all \(i \neq v^+,w^+ \), \( x^i \equiv_\mu y^i \).
  \end{itemize}and we set \( \alg{C} = \alg{C}[\mathbb{G}] = \alg{D}/\Theta \). We will usually write, e.g., \( \Gamma \) instead of \( \Gamma / \Theta \) when context makes it unambiguous.

  \begin{claim}
    \( \Theta_{|\Gamma} \) consists of just the generating pairs and no more.
  \end{claim}

  To see this, let \( \langle \aI{x}, \aI{y} \rangle \) be a generating pair and \( \aI{f} \in \Pol{1}{\alg{D}_{|U}} \). Then \( \langle \aI{f}(\aI{x}), \aI{f}(\aI{y}) \rangle \) is clearly a generating pair if \(\aI{f}\) collapses \(\mu\) to points, or if \( \aI{x} \) and \( \aI{y} \) belong to \( \Gamma_0 \), so let \( \aI{f} \in \PolGrp{\alg{D}}{U} \). Then if \( \aI{x}, \aI{y} \in \Gamma_V \) with \( \mathrm{Label}(\aI{x}) = \mathrm{Label}(\aI{y}) = \langle v,a \rangle \) then \begin{align*}
    a = x^{v^+} = y^{v^+} &\Rightarrow f^v(a) = f^v(x^{v^+}) = f^v(y^{v^+}) \\
    x^i \equiv_\mu y^i &\Rightarrow f^i(x^i) \equiv_\mu f^i(y^i)
  \end{align*}so \( \langle \aI{f}(\aI{x}), \aI{f}(\aI{y}) \rangle \) is again a generating pair. The proof for generating pairs from \( \Gamma_E \) is identical.

  By Lemma \ref{lemma:sigma_definable}, \( \sigma^I \) is a uniformly definable congruence; it follows that quantification over any of the groups \( \TwinGrp{\alg{D}}{\Gamma}, \PolGrp{\alg{D}}{\Gamma}, \TwinGrp{\alg{C}}{\Gamma}, \PolGrp{\alg{C}}{\Gamma} \) is uniformly first-order in the respective algebra. Of course, we also have that \( \Gamma = U^I \) (respectively \( U^I / \Theta \)) is a definable subset of both algebras as well, since it consists of precisely the fixed points of the polynomial retraction \(e\).

  \begin{claim}
    \begin{itemize}
      \item If \(g \in \TwinGrp{\alg{S}}{U} \), \(a \equiv_\sigma b\) and \(g(a) \equiv_\mu a \) then \(g(b) \equiv_\mu b \).
      \item \( \Gamma_0 \cap U^I \) is uniformly definable (using at most \(n \cdot \card{\TwinGrp{\alg{S}}{U}}\) parameters) in \( \alg{C} \).
    \end{itemize}
  \end{claim}

  The first part is true because \( \TC{\sigma}{\sigma}{\mu} \): \begin{align*}
    a = T(a,\vec{d}_{\mathrm{id}}) &\equiv_\mu T(a,\vec{d}_g) = g(a) \\ &\Downarrow \\
    b = T(b,\vec{d}_{\mathrm{id}}) &\equiv_\mu T(b,\vec{d}_g) = g(b)
  \end{align*}where \(T(v_0, \ldots, v_n) \) is the universal term for \( \TwinGrp{\alg{S}}{U} \) defined above. To the second: for each \( g \in \TwinGrp{\alg{S}}{U} \) let \( \vec{\aI{c}}_g \) be constants so that \[ T(\cdot, \vec{\aI{c}}_g) = \mathrm{id}_{|V} \oplus g_{|\infty} \]Then for \( \aI{x} \in U^I \) we have \[ T(\aI{x},\vec{\aI{c}}_g) \equiv_\Theta \mathrm{x} \; \iff \; g(x^\infty) \equiv_\mu x^\infty \]Hence \[\aI{x} \in \Gamma_0 \; \Rightarrow \; \forall g \in \TwinGrp{\alg{S}}{U} \;\left( T(\aI{x},\vec{\aI{c}}_g) \equiv_\Theta \aI{x} \; \rightarrow \; T(\aI{x},\vec{\aI{d}}_g) \equiv_\Theta \aI{x} \right) \](where \( \vec{\aI{d}}_g \) are the obvious diagonal elements), while if \( \aI{x} \in \Gamma_V \) (resp. \( \Gamma_E \)) with label \( \langle v,a \rangle \) (resp. \( \langle v,a,w,b \rangle \)) and \(g(a) \equiv_{\mu \setminus \bot} a \) then \( g(x^\infty) \equiv_\mu x^\infty \) so \[ T(\aI{x},\vec{\aI{c}}_g) \equiv_\Theta \aI{x} \text{ and } T(\aI{x},\vec{\aI{d}}_g) \not \equiv_\Theta \aI{x} \]This proves the claim.

  We are almost done: for the last step, define a preorder \(\ll\) on \( \Gamma \setminus \Gamma_0 \) by \begin{align*}
    \aI{x} \ll \aI{y} \iff &\exists f \in \PolGrp{\alg{S}}{U} \: \exists \aI{g} \in \TwinGrp{\alg{C}}{\Gamma} \: \left[ \aI{g} f (\aI{x}) \equiv_\sigma \aI{y} \: \& \right. \\ &\quad \left. \forall \aI{h} \in \TwinGrp{\alg{C}}{\Gamma} \: \left[ \aI{h} \aI{g} f (\aI{x}) \not \equiv_\Theta \aI{g} f (\aI{x}) \rightarrow \aI{h}(\aI{y}) \not \equiv_\Theta \aI{y} \right] \right] \end{align*}

  \begin{claim} 
    \begin{enumerate}
    \item If \( \aI{x}, \aI{y} \in \Gamma_V \) (resp. \( \Gamma_E \)) are labeled by the same vertex (resp. edge), they are \( \ll \)-equivalent.
    \item If \( \aI{x}, \aI{y} \in \Gamma_V \) (resp. \( \Gamma_E \)) are labeled by different vertices (resp. edges), they are \( \ll \)-incomparable.
    \item If \( \aI{x} \in \Gamma_E \) and \( \aI{y} \in \Gamma_V \) then \( \aI{x} \not \ll \aI{y} \).
    \item If \( \aI{x} \in \Gamma_V, \aI{y} \in \Gamma_E \), then \( \aI{x} \ll \aI{y} \) iff \( \aI{x} \) is labeled by one of the endpoints of the edge which labels \( \aI{y} \).
    \end{enumerate}
  \end{claim}

  This claim will complete the proof of the theorem, since up to \( \ll \)-biequivalence, vertices of \( \mathbb{G} \) correspond precisely to \( \ll \)-classes at level zero, edges to classes at level one, and two vertices are joined iff there is a class properly dominating both.

  \begin{enumerate}
  \item Say \( \aI{x}, \aI{y} \in \Gamma_E \), \( \mathrm{Label}(\aI{x}) = \langle v,a_1,w,b_1 \rangle \) and \( \mathrm{Label}(\aI{y}) = \langle v,a_2,w,b_2 \rangle \). Then \( \mathcal{O}(a_j) = \mathcal{O}(b_j) \) (\(j \in \{1,2 \} \)), and we can choose \( f \in \PolGrp{\alg{S}}{U} \) so that \(f \mathcal{O}(a_1) = \mathcal{O}(a_2) \). Then we can choose \( \{g^i\}_{i \in I} \in \TwinGrp{\alg{S}}{U} \) so that \( g^v f(a_1) = a_2 \), \( g^w f(b_1) = b_2 \), and \( g^i f(x^i) \equiv_\mu y^i \) for all other \(i\), and set \( \aI{g} = \bigoplus g^i \). Then in fact \(\aI{g} f (\aI{x}) \equiv_\Theta \aI{y} \) so \( \aI{x} \ll \aI{y} \) holds automatically. The proof is the same for \( \Gamma_V \) except easier.
  \item Say \( \aI{x} \) has a spike at a V-block where \( \aI{y} \) does not, say at \(v\). Then for every \( f \in \PolGrp{\alg{S}}{U} \) and every \( \aI{g} \in \TwinGrp{\alg{C}}{\Gamma} \), \(\aI{g} f (\aI{x}) \) has a spike at \(v\), which \(\aI{y}\) does not. Assume \( \aI{g} f (\aI{x}) \equiv_\sigma \aI{y} \). Choose \( h \in \TwinGrp{\alg{S}}{U} \) such that \( h g^v f (x^{v^+}) \equiv_{\mu \setminus \bot} g^v f (x^{v^+}) \); then \( h(y^{v^+}) \equiv_\mu y^{v^+} \). Let \( \aI{h} \in \TwinGrp{\alg{C}}{\Gamma} \) be \(h\) on \( \{v^\pm\} \) and the identity on all other vertex blocks; then \[ \aI{hg} f (\aI{x}) \not \equiv_\Theta \aI{g} f (\aI{x}) \; \& \: \aI{h}(\aI{y}) \equiv_\Theta \aI{y} \]
  \item The same as in (b).
  \item The direction (\( \Rightarrow \)) is the same as in (b). For (\( \Leftarrow \)), assume that \( \mathrm{Label}(\aI{x}) = \langle v,a_1 \rangle \), \(\mathrm{Label}(\aI{y}) = \langle v,a_2,w,b \rangle \). Choose \( f \in \PolGrp{\alg{S}}{U} \) with \( f(a_1) = a_2 \), and for \(i \neq v^{\pm}\) choose \( g^i \in \TwinGrp{\alg{S}}{U} \) so that \(g^i f(x^i) \equiv_\mu y^i \), \(g^v = \mathrm{id} \), \( \aI{g} = \bigoplus_i g^i \); then we have \(\aI{z} := \aI{g}f(\aI{x}) \equiv_\mu \aI{y} \) and \( z^{v^+} = y^{v^+} \). Consequently, if \( \aI{h}(\aI{z}) \not \equiv_\Theta \aI{z} \) then either \[ h^i(z^i) \not \equiv_\mu z^i \] for some \(i \in I \), in which case \(h^i(y^i) \not \equiv_\mu y^i \), or \[ h^v(y^{v^+}) = h^{v}(z^{v^+}) \neq z^{v^+} = y^{v^+}\]so in either case \(\aI{h}(\aI{y}) \not \equiv_\Theta \aI{y} \).
  \end{enumerate}The claim and the Lemma are proven.
\end{proof}

\begin{lemma}\label{lemma:C(sigma,sigma;bot)}
  Let \( \alg{S} \) be a finite subdirectly irreducible algebra with unary-type monolith \(\mu\) and strongly solvable radical \(\sigma\) satisfying \( \TC{\sigma}{\sigma}{\mu} \), \( \TC{\sigma}{\mu}{\bot} \), and \( \TC{\mu}{\sigma}{\bot} \) but not \( \TC{\sigma}{\sigma}{\bot} \). Then \( \HSP{\alg{S}} \) is hereditarily finitely undecidable.
\end{lemma}

\begin{proof}
  Choose a package \begin{align*}
    c = t_0(a_0,\vec{b}_0) &= t_0(a_0,\vec{b}_1) \\ &\text{but} \\
    m_0 = t_0(a_1,\vec{b}_0) &\neq t_0(a_1,\vec{b}_1) = m_1
  \end{align*}witnessing \(\neg \TC{\sigma}{\sigma}{\bot} \), where \(a_0 / \sigma = a_1 / \sigma =: A \) and \(\vec{b}_0 \equiv_\sigma \vec{b}_1 \). Since \( \TC{\sigma}{\sigma}{\mu} \), \(m_0 \equiv_\mu m_1 \), and we may suppose that the range of \( t_0(v_0, \ldots, v_\ell) \) is included in a \( (\bot,\mu) \)-minimal set \(U\). Denote the trace containing the \(m_j\) by \(M\).

  We will be working with diagonal subpowers \( \alg{X} \leq \alg{S}^I \) and their quotients \(\alg{Y} = \alg{X}/\Theta \), where \( \Theta \leq \sigma = \sigma^I \cap X^2 \in \Con{\alg{X}} \). Lemma \ref{lemma:sigma_definable} once again implies that \( \sigma \) is a definable congruence in all such \( \alg{Y} \).

  We will wherever possible refer to elements of \(\alg{Y}\) with \( \aI{x} \) rather than \( \aI{x} / \Theta \), with the understanding that \( \aI{x} \in S^I \) is one representative. (Of course, this will necessitate showing that certain properties are well-defined.)

  For such algebras \( \alg{Y} \), and \( \vec{\aI{y}}_1, \vec{\aI{y}}_2 \equiv_\sigma \vec{\aI{b}}_0 \) define \[ \Eset{\alg{Y}}{\vec{\aI{y}}_1}{\vec{\aI{y}}_2} = \left\{ \aI{x} \equiv_\sigma \aI{a}_0 \colon \alg{Y} \models t_0(\aI{x},\vec{\aI{y}}_1) = t_0(\aI{x},\vec{\aI{y}}_2) \right\} \]In particular, we have \[ \Eset{\alg{S}}{\vec{b}_0}{\vec{b}_1} \subsetneq A \]and there is no loss of generality in assuming that the \( \vec{b}_j \) are chosen so that their equalizer set is maximal for being properly included in \(A\).

  We will be using \(\ell\)-tuples extensively, so to avoid a proliferation of vector notation we will reserve the letters \(b, y, z\) for \(\ell\)-tuples and \(a,x\) for single elements.
	
  The plan is as follows: We want to interpret the class of graphs with at least three vertices into \( \HSP{\alg{S}} \). Given such a graph \( \mathbb{G} = \langle V, E \rangle \), we will choose an index set \(I\) and a diagonal subpower \( \alg{D} \leq \alg{S}^I \), which will depend only on \(V\), and then a congruence \( \Theta \in \Con{\alg{D}} \) below \( \sigma^I \) (in fact, below \( \mu^I \)), which will depend on both \(V\) and \(E\), and set \( \alg{C} = \alg{D} / \Theta \). \(\Theta\) will be sparse in a sense we will make precise. Then we will define a set \( \mathcal{B} \subset C^\ell \), and show that a preorder \(\ll\) recovering the index set \(I\) is definable there; vertices will interpret as unions of two \(\ll\)-biequivalence classes, and the edge relation from \(\mathbb{G}\) will be first-order definable on these vertices. Here ``definable'' will include reference to \( |A| + 1 \) parameters (in addition to the diagonal).

  We begin with a graph \( \mathbb{G} = \langle V, E \rangle \), and set \(I = V^\pm \sqcup \{\infty\}\) as in Lemma \ref{lemma:twins act trivially}. Define \( \alg{D} \leq \alg{S}^I \) to be the subalgebra consisting of all elements which are constant modulo \( \sigma \). By the same logic applied in Claim \ref{claim:structure of T}, \( \PolGrp{\alg{D}}{U} \) consists of those \( \aI{f} \in (\PolGrp{\alg{S}}{U})^I \) such that all \( f^i \) belong to the same coset modulo \( \TwinGrp{\alg{S}}{U} \). (Here the coordinate functions \( f^{v^+},f^{v^-} \) may be different.) The relation \( \TC{\sigma}{\mu}{\bot} \) implies that a polynomial \( \aI{f}(v_0) = t(v_0, \vec{\aI{d}}) \) whose image is contained in \(U^I\) is either a permuation of \(U\) at all coordinates or collapses traces to points at all coordinates. We note for future reference that
	
  \begin{claim}\label{claim:twins act the same}if \(f_1,f_2 \in \PolGrp{\alg{S}}{U}\) belong to the same coset modulo \( \TwinGrp{\alg{S}}{U} \), and if \( f_1(M) = M = f_2(M) \) then by Lemma \ref{lemma:twins act trivially} \({f_1}_{|M} = {f_2}_{|M} \)
  \end{claim}In particular, this is true if these are the coordinate functions of some \( \aI{f} \in \PolGrp{\alg{D}}{U} \).

  Let \( \alg{C} = \alg{D} / \Theta \), where \(\Theta\) is the congruence on \( \alg{D} \) generated by identifying \begin{align*}
    {m_1}_{|\{v^+\}} \oplus {m_0}_{|I \setminus \{v^+\}} &\equiv_\Theta {m_1}_{|v^-} \oplus {m_0}_{|I \setminus \{v^-\}} & (v \in V) \\
    {m_1}_{|\{v^+,w^+ \} } \oplus {m_0}_{|I \setminus \{v^+,w^+ \}} &\equiv_\Theta {m_1}_{|\{v^-,w^- \} } \oplus {m_0}_{|I \setminus \{v^-,w^- \}} & (v \Edge w)
  \end{align*}

  \begin{claim}
    \begin{enumerate}
    \item \(\Theta \leq \mu^I \), and if \( \aI{x}_1 \equiv_\Theta \aI{x}_2 \) then \( x_1^\infty = x_2^\infty \).
    \item \( \Theta_{|U^I} \) has blocks of cardinality 1 and 2 only.
    \item If \( \aI{x}_1, \aI{x}_2 \in U^I \) and \( \aI{x}_1 \equiv_\Theta \aI{x}_2 \), then the set of coordinates where they differ is either empty, one \(V\)-block \( \{v^+,v^-\} \), or two \(V\)-blocks \( \{ v^+, v^-, w^+, w^- \} \) where \( v \Edge w \).
    \end{enumerate}	
  \end{claim}

  The first statement is clear. To see (b), first observe that if \( \aI{f} \in \Pol{1}{\alg{D}_{|U}} \setminus \PolGrp{\alg{D}}{U} \) then \[ \aI{f}({m_1}_{|\{v^+\}} \oplus {m_0}_{|I \setminus \{v^+\}}) = \aI{f}( {m_0}_{|I} ) = \aI{f}({m_1}_{|\{v^-\}} \oplus {m_0}_{|I \setminus \{v^-\}}) \]so it suffices to consider images of generating pairs under permutations \( \aI{f} \in \PolGrp{\alg{D}}{U} \). Next, since \( \PolGrp{\alg{S}}{U} / \TwinGrp{\alg{S}}{U} \) acts on orbits and since we may assume that \(\mathcal{O}(m_0) \neq \mathcal{O}(m_1)\), we may conclude that any image \( \aI{f}({m_1}_{|\{v^+\}} \oplus {m_0}_{|I \setminus \{v^+\}}) \) takes values in one orbit at all coordinates except \( v^+ \) and in a different orbit there, and similarly for the other elements involved in the generating pairs. We prove the claim for generators of the vertex type; the edge-type argument is no different.

  Given any putative \( \Theta_{|U} \)-block of more than two elements, we can find a subset of three elements of the form \begin{align*}
    \aI{x}_1 = \aI{f}_1 ({m_1}_{|\{v^+\}} \oplus {m_0}_{|I \setminus \{v^+\}}) &= \aI{f}_2 ({m_1}_{|\{v^+\}} \oplus {m_0}_{|I \setminus \{v^+\}}) = \aI{x}_2 \\
    \aI{y}_1 = \aI{f}_1 ({m_1}_{|\{v^-\}} \oplus {m_0}_{|I \setminus \{v^-\}}) &\stackrel{?}{=} \aI{f}_2 ({m_1}_{|\{v^-\}} \oplus {m_0}_{|I \setminus \{v^-\}}) = \aI{y}_2
  \end{align*}or vice versa. The first line shows that \( \aI{f}_2^{-1} \circ \aI{f}_1 (M^I) = M^I\); but since \( \aI{f}_2^{-1} \circ \aI{f}_1 \in \TwinGrp{\alg{D}}{U} \), it must fix \(M^I\) pointwise. Hence \( \aI{y}_1 = \aI{y}_2 \).

  Looking a little more closely at the argument, we see that in fact a pair of unequal elements \( \aI{x}_1, \aI{x}_2 \in U^I\) are \(\Theta\)-related iff they are the image of a generating pair under some \( \aI{f} \in \PolGrp{\alg{D}}{U} \). Claim (c) follows immediately.

  With this claim in hand, it is well-defined to speak of \( x^\infty \) for \( \aI{x} \in C \). Furthermore, by Claim \ref{claim:twins act the same}, the image of any member of a generating pair under \( \aI{f} \in \PolGrp{\alg{D}}{U} \) cannot be a constant element. (In other words, the constant elements of \( U^I \) are isolated modulo \( \Theta \).)

\addtocounter{claim}{1}
  Throughout the remainder of the proof, any \(\ell\)-tuple \(\aI{y}\) or \(\aI{z}\) will be assumed to be \(\sigma\)-congruent to \(\aI{b}_0\), and to satisfy the condition \[\aI{c} = t_0(\aI{a}_0,\aI{b}_0) \equiv_\Theta t_0(\aI{a}_0,\aI{y}) \tag{\arabic{section}.\arabic{thm}.\arabic{claim}}\label{eq:E(b0,y)} \](which is clearly first-order in \( \alg{C} \)). Since \(\aI{c}\) is isolated, this is in fact an equality. (For instance, every \(\ell\)-tuple from \( \{b_0,b_1\}^I \) satisfies this condition, and our life would be much easier if we could work with just that set. The following can be read as a way of coming as close to this as feasible.)

  \begin{claim}
    Define a parameter \( \mathfrak{b} = {b_1}_{|\infty} \oplus {b_0}_{|I \setminus \{ \infty \}} \) which will be fixed throughout the remainder of the proof. The predicates \[ \Eset{\alg{S}}{y^\infty}{b_0} = A \] and \[ \Eset{\alg{S}}{y^\infty}{b_1} = A \](in the free variable \( \aI{y}\)) are definable using \(\mathfrak{b}\) together with \( |A| \) other parameters.
  \end{claim}

  This is because \begin{align*}
    \Eset{\alg{S}}{y^\infty}{b_0} = A &\iff \bigwedge_{a \in A} a_{|\infty} \oplus {a_0}_{|I \setminus \{ \infty \}} \in \Eset{\alg{C}}{\aI{b}_0}{\aI{y}} \\
    \Eset{\alg{S}}{y^\infty}{b_1} = A &\iff \bigwedge_{a \in A} a_{|\infty} \oplus {a_0}_{|I \setminus \{ \infty \}} \in \Eset{\alg{C}}{\aI{b}_1}{\aI{y}}
  \end{align*}

  We will not name or even make explicit mention of the parameters \( a_{|\infty} \oplus {a_0}_{|I \setminus \{ \infty \}} \) any more, but they are implicitly present in all that follows. 
	
  The next claim does most of the heavy lifting in this lemma.
	
  \begin{claim}\label{claim:rectangular Psets}
    Suppose \(\aI{y}\) satisfies condition (\ref{eq:E(b0,y)}) and that \( \Eset{\alg{S}}{y^\infty}{b_1} = A \). Then the set \[ P(\aI{y}) := \left( \bigoplus_{i \neq \infty} \Eset{\alg{S}}{b_0}{y^i} \oplus \left( A \setminus \Eset{\alg{S}}{b_0}{b_1} \right) \right) / \Theta \]is a definable subset of \(\alg{C}\).
  \end{claim}

  To show this, we will need one auxiliary definition which will be repeatedly useful:
	
  \begin{defn*}
    If \( \Eset{\alg{S}}{y_1^\infty}{b_0} = A = \Eset{\alg{S}}{y_2^\infty}{b_1} \), write \( \aI{y}_1 \propto \aI{y}_2 \) if the following equivalent conditions are satisfied: \begin{enumerate}
    \item \(\Eset{\alg{S}}{y_1^i}{y^i_2} = A \) for all \( i \neq \infty \) \label{eq:propto_a}
    \item \( \Eset{\alg{C}}{\aI{b}_0}{\mathfrak{b}} \subseteq \Eset{\alg{C}}{\aI{y}_1}{\aI{y}_2} \) \label{eq:propto_b}
    \end{enumerate}
  \end{defn*}

  To see that these conditions are in fact equivalent, in the direction \( \eqref{eq:propto_a} \Rightarrow \eqref{eq:propto_b} \), if \( t_0(\aI{x},\aI{b}_0) \equiv_\Theta t_0(\aI{x},\mathfrak{b}) \), then \[ t_0(x^\infty,y_1^\infty) = t_0(x^\infty,b_0) = t_0(x^\infty,b_1) = t_0(x^\infty,y_2^\infty) \]so that \( t_0(\aI{x},\aI{y}_1) \) is in fact equal to \( t_0(\aI{x},\aI{y}_2) \). Conversely, fix \( i \neq \infty \) and \(a \in A \). We know that \begin{align*}
    \aI{c} = t_0(\aI{a}_0,\aI{y}_1) &= t_0(\aI{a}_0,\aI{y}_2) \\
    &\text{and} \\ t_0({a}_{|i} \oplus {a_0}_{|I \setminus \{i\}},\aI{b}_0) &= t_0({a}_{|i} \oplus {a_0}_{|I \setminus \{i\}},\mathfrak{b}) \\ &\text{hence} \\
    t_0({a}_{|i} \oplus {a_0}_{|I \setminus \{i\}},\aI{y}_1) &\equiv_\Theta t_0({a}_{|i} \oplus {a_0}_{|I \setminus \{i\}},\aI{y}_2)
  \end{align*}and these elements do not differ except possibly at \(i\); hence they are in fact equal, showing that \[ t_0(a,y_1^i) = t_0(a,y_2^i) \]Note that condition \eqref{eq:propto_b} is clearly first-order.

  Now to the proof of Claim \ref{claim:rectangular Psets}: let \( \aI{y} \) be as in the statement, and let \(\aI{z}\) be the tuple which agrees with \(b_0\) at \(\infty\) and with \(\aI{y}\) everywhere else, so \( \aI{z} \propto \aI{y} \). 
	
  Now assume further that \( \aI{x} \in P(\aI{y}) \). Then \begin{align*}
    t_0(\aI{x},\aI{b}_0) &= t_0(\aI{x},\aI{z}) \text{ and}\\
    t_0(\aI{x},\mathfrak{b}) &= t_0(\aI{x},\aI{y}) \text{ and}\\
    t_0(\aI{x},\aI{b}_0) &\not \equiv_\Theta t_0(\aI{x},\mathfrak{b})
  \end{align*}
	
  We have shown \begin{align*}
    \aI{x} \in P(\aI{y}) \Rightarrow \exists \aI{z} \equiv_\sigma \aI{b}_0 \; &\Eset{\alg{S}}{z^\infty}{b_0} = A \text{ and } \aI{z} \propto \aI{y} \text{ and}\\
    &\aI{x} \in \Eset{\alg{C}}{\aI{b}_0}{\aI{z}} \cap \Eset{\alg{C}}{\mathfrak{b}}{\aI{y}} \text{ and}\\
    &\aI{x} \notin \Eset{\alg{C}}{\aI{b}_0}{\mathfrak{b}}
  \end{align*}Next, we show that the converse holds as well.

  Assume the following: \begin{align}\stepcounter{claim}
    \aI{x} &\in A^I \text{ but not in } P(\aI{y}) \tag{\EQnum}\label{eq:z.1}\\ \stepcounter{claim}
    \aI{z} &\equiv_\sigma \aI{b}_0 \text{ with } \Eset{\alg{S}}{z^\infty}{b_0} = A \tag{\EQnum} \label{eq:z.2}\\ \stepcounter{claim}
    \aI{z} &\propto \aI{y} \tag{\EQnum}\label{eq:z.3} \\ \stepcounter{claim}
    \aI{x} &\in \Eset{\alg{C}}{\aI{b}_0}{\aI{z}} \cap \Eset{\alg{C}}{\mathfrak{b}}{\aI{y}} \tag{\EQnum} \label{eq:z.4}		
  \end{align}We must show that \( \aI{x} \in \Eset{\alg{C}}{\aI{b}_0}{\mathfrak{b}} \).

  By (\ref{eq:z.1}), we know that for some \(i \neq \infty \), \(x^i \notin \Eset{\alg{S}}{b_0}{y^i} \). By (\ref{eq:z.3}), \( \Eset{\alg{S}}{y^i}{z^i} = A \) for all \(i \neq \infty \).
	
  Working in \(\alg{D}\), define elements \begin{equation*} \begin{matrix}
      \aI{u}_{00} = t_0(\aI{x},\aI{b}_0) & t_0(\aI{x},\aI{z}) = \aI{u}_{01} \\
      \aI{u}_{10} = t_0(\aI{x},\mathfrak{b}) & t_0(\aI{x},\aI{y}) = \aI{u}_{11} 		
  \end{matrix} \end{equation*}Our assumptions imply the following: \begin{align} \stepcounter{claim}
    u_{10}^\infty &= u_{11}^\infty &\text{ since } \Eset{\alg{S}}{y^\infty}{b_1} = A \tag{\EQnum} \label{eq:u.1}\\ \stepcounter{claim}
    u_{00}^\infty &= u_{01}^\infty &\text{ by } (\ref{eq:z.2}) \tag{\EQnum} \label{eq:u.2}\\ \stepcounter{claim}
    i \neq \infty \Rightarrow u_{00}^i &= u_{10}^i &\text{ (obvious)} \tag{\EQnum} \label{eq:u.3}\\ \stepcounter{claim}
    i \neq \infty \Rightarrow u_{01}^i &= u_{11}^i &\text{ by } (\ref{eq:z.3}) \tag{\EQnum} \label{eq:u.4}\\ \stepcounter{claim}
    \aI{u}_{00} &\equiv_\Theta \aI{u}_{01} &\text{ by } (\ref{eq:z.4}) \tag{\EQnum} \label{eq:u.5}\\  \stepcounter{claim}
    \aI{u}_{10} &\equiv_\Theta \aI{u}_{11} &\text{ by } (\ref{eq:z.4}) \tag{\EQnum} \label{eq:u.6}\\ \stepcounter{claim}
    \aI{u}_{10} &\neq \aI{u}_{11} &\text{ by } (\ref{eq:z.1}) \tag{\EQnum} \label{eq:u.7}
  \end{align}Together, these imply that \( \aI{u}_{00} \neq \aI{u}_{01} \) also.

  Choose \( \aI{f} \in \PolGrp{\alg{D}}{U} \) so that \( \{ \aI{f}(\aI{u}_{10}),\aI{f}(\aI{u}_{11}) \} \) is a generating pair for \(\Theta\), and let \(\aI{w}_{ij} = \aI{f}(\aI{u}_{ij}) \). Then (\ref{eq:u.1})-(\ref{eq:u.7}) are still true of the \(\aI{w}_{ij}\). By definition, \( \aI{w}_{10}, \aI{w}_{11} \in M^I \); the same is true of \( \aI{w}_{00}, \aI{w}_{01} \), which is shown as follows: for \( i \neq \infty \), \( w_{0j}^i = w_{1j}^i \in M \), while at \(\infty\) we can use \( \TC{\sigma}{\sigma}{\mu} \) to get \begin{align*}
    f^\infty t_0 (a_0,b_0) &= f^\infty t_0(a_0,b_1) \\ &\Downarrow \\
    w_{01}^\infty = w_{00}^\infty = f^\infty t_0 (x^\infty,b_0) &\equiv_\mu f^\infty t_0 (x^\infty,b_1) = w_{10}^\infty \in M
  \end{align*}

  Similarly, we may choose \( \aI{g} \in \PolGrp{\alg{D}}{U} \) so that \( \{ \aI{g}(\aI{w}_{00}),\aI{g}(\aI{w}_{01}) \} \) is a generating pair for \( \Theta \), whose nontriviality is guaranteed by (\ref{eq:u.7}). But we have \( g^i(M) = M\) for all \( i \in I \), so we may assume (by Claim \ref{claim:twins act the same}) that \( g^i = g^j = g \) for all \(i,j \in I \).

  Now: since \( \{ \aI{w}_{10}, \aI{w}_{11} \} \) form a generating pair for \(\Theta\) and since \( |V| \geq 3 \), there exists \(v \in V \) so that \(w_{10}^{v^+} = w_{11}^{v^+} \). This value cannot be \(m_1\), so we have \[ w_{00}^{v^+} = w_{10}^{v^+} = m_0 = w_{11}^{v^+} = w_{01}^{v^+} \]Hence \[ g(w_{00}^{v^+}) = g(m_0) = g(w_{01}^{v^+}) \]which implies \(g(m_0) = m_0 \) (since \( \{ \aI{g}(\aI{w}_{00}), \aI{g}(\aI{w}_{01}) \} \) are a generating pair). But then \begin{align*}
    \left( \aI{g}(\aI{w}_{00}) \right)^\infty &= m_0 = \left( \aI{g}(\aI{w}_{01}) \right)^\infty \\ &\Downarrow \\
    w_{00}^\infty &= m_0 = w_{01}^\infty = w_{10}^\infty = w_{11}^\infty \\ &\Downarrow \\
    \aI{w}_{00} &= \aI{w}_{10} \\ &\Downarrow \\ 
    \aI{u}_{00} &= \aI{u}_{10} \\ &\Downarrow \\ 
    \aI{x} &\in \Eset{\alg{C}}{\aI{b}_0}{\mathfrak{b}}
  \end{align*}This completes the proof of Claim \ref{claim:rectangular Psets}.

  The foregoing claim implies that the mapping \[\aI{y} \mapsto \Eset{\alg{S}}{b_0}{y^i} \]on the set of those points \( \aI{y} \equiv_\sigma \aI{b}_0 \) such that \[ \Eset{\alg{S}}{y^\infty}{b_1} = A \]is invariant modulo \(\Theta\). Let \( \aI{y} \) be such a point. For any \( a \in A \), \( \aI{a} \in \Eset{\alg{C}}{\mathfrak{b}}{\aI{y}} \) iff \( a \) belongs to all the factor sets \( \Eset{\alg{S}}{b_0}{y^i} \) (\(i \neq \infty \)). It follows that the set \( \mathcal{B} \) of those \( \aI{y} \) such that \[ \Eset{\alg{S}}{b_0}{b_1} \subseteq \Eset{\alg{S}}{b_0}{y^i} \text{ for all } i \neq \infty \text{ and } \Eset{\alg{S}}{y^\infty}{b_1} = A \]that is, those \( \aI{y} \) such that \[\Eset{\alg{S}}{b_0}{y^i} \in \left\{\Eset{\alg{S}}{b_0}{b_1}, A \right\} \text{ for all } i \neq \infty \]is definable (by asserting that \( \aI{a} \in \Eset{\alg{C}}{\mathfrak{b}}{\aI{y}} \) for each \(a \in \Eset{\alg{S}}{b_0}{b_1} \)). We may define a preorder on \( \mathcal{B} \) by \[ \aI{y}_1 \ll \aI{y}_2 \iff P(\aI{y}_2) \subseteq P(\aI{y}_1) \](Note the reverse inclusion.) Because we chose \( \Eset{\alg{S}}{b_0}{b_1}\) maximal, the associated partial order is isomorphic to the boolean algebra with \( 2|V| \) atoms. Indeed, each tuple \( {b_1}_{|i,\infty} \oplus {b_0}_{|I \setminus \{i,\infty\}} \) sits at \(\ll\)-level 1; we denote the elements at \(\ll\)-levels one and two by \( \mathcal{B}_1 \) and \( \mathcal{B}_2 \) respectively. Let \( \mathrm{WHO}(v_0) \) be a formula (in the parameters we have already mentioned) asserting that \(v_0 \in \mathcal{B}_1 \).

  For \( \aI{y} \in \mathcal{B}_1 \), let \( \chi(\aI{y}) \) denote the (unique) coordinate \(i \neq \infty\) such that \( \Eset{\alg{S}}{b_0}{y^i} = \Eset{\alg{S}}{b_0}{b_1} \subsetneq A \). If \( \chi(\aI{y}) \in \{v^+,v^-\} \) we set \( |\chi|(\aI{y}) = v \).

  Assume that \( |\chi|(\aI{y}_1) = |\chi|(\aI{y}_2) \). Then either \( \chi(\aI{y}_1) = \chi(\aI{y}_2) \), which we know to be definable, or for some \( v \in V \) we have \( \chi(\aI{y}_1) = v^+ \) and \( \chi(\aI{y}_2) = v^- \) (or vice versa). Define \[ \aI{b}^+ = {b_1}_{|v^+, \infty} \oplus {b_0}_{|\mathrm{else}} \qquad \aI{b}^- = {b_1}_{|v^-, \infty} \oplus {b_0}_{|\mathrm{else}} \]Then \( \aI{b}^+, \aI{b}^- \in \mathcal{B}_1\), \( \chi(\aI{y}_1) = \chi(\aI{b}^+) \), and \( \chi(\aI{y}_2) = \chi(\aI{b}^-) \). Next define \[ \aI{z}^+ = {b_1}_{|v^+} \oplus {b_0}_{|\mathrm{else}} \qquad \aI{z}^- = {b_1}_{|v^-} \oplus {b_0}_{|\mathrm{else}} \]Then \( \aI{z}^+ \propto \aI{b}^+ \), \( \aI{z}^- \propto \aI{b}^- \), and \[ t_0 (\aI{a}_1, \aI{z}^+) = {m_1}_{|v^+} \oplus {m_0}_{|\mathrm{else}} \equiv_\Theta {m_1}_{|v^-} \oplus {m_0}_{|\mathrm{else}} = t_0 (\aI{a}_1, \aI{z}^-)\]We have shown that for \( \aI{y}_1, \aI{y}_2 \in \mathcal{B}_1 \), \begin{align*}
    |\chi|(\aI{y}_1) = |\chi|(\aI{y}_2) \Rightarrow \alg{C} \models \: &\chi(\aI{y}_1) = \chi(\aI{y}_2) \text{ or } \\
    &\exists v_3, v_4, v_5, v_6, \mathrm{WHO}(v_3) \: \& \: \mathrm{WHO}(v_4) \: \&   \\
    & \quad \chi(\aI{y}_1) = \chi(v_3) \: \& \: \chi(\aI{y}_2) = \chi(v_4) \: \& \\
    & \quad v_5 \propto v_3 \: \& \: v_6 \propto v_4 \: \& \\
    & \quad t_0(a_1,v_5) = t_0(a_1,v_6)
  \end{align*}Let the last formula be denoted \(\mathrm{EQ}(\aI{y}_1,\aI{y}_2)\), with the understanding that the variables \(v_3\) through \(v_6\) are really \(\ell\)-tuples.

  \begin{claim}
    The converse holds too; that is, the formula \( \mathrm{EQ}(v_1,v_2) \) defines the equivalence relation \( |\chi|(v_1) = |\chi|(v_2) \) on \(\mathcal{B}_1\).
  \end{claim}

  To show this, let \( \chi(\aI{y}_1) = v^+ \), say, and \( \chi(\aI{y}_2) \notin \{ v^+, v^- \} \); we must show \( \neg \mathrm{EQ}(\aI{y}_1, \aI{y}_2) \). To this end, let \( \aI{y}_3, \aI{y}_4 \in \mathcal{B}_1 \) with \( \chi(\aI{y}_3) = \chi(\aI{y}_1) \), \( \chi(\aI{y}_4) = \chi(\aI{y}_2) \), and let \( \aI{z}_5 \propto \aI{y}_3 \), \( \aI{z}_6 \propto \aI{y}_4 \). Then for \( i \neq \infty \), \( t_0(a_1, z_5^i) = t_0(a_1, y_3^i) \) and \( t_0(a_1, z_6^i) = t_0(a_1, y_4^i) \) by the definition of the relation \( \propto \). Hence \begin{align*}
    \text{If } i = \infty \text{ then } &t_0(a_1, z_5^i) = t_0(a_1, b_0) = t_0(a_1, z_6^i) \\
    \text{If } i = v^+ \text{ then } &t_0(a_1, z_5^i) = t_0(a_1, y_3^i) \neq t_0(a_1, b_0) = t_0(a_1, y_4^i) = t_0(a_1, z_6^i) \\
    \text{If } i = \chi(\aI{y}_2) \text{ then } &t_0(a_1, z_5^i) = t_0(a_1, y_3^i) = t_0(a_1, b_0) \neq t_0(a_1, y_4^i) = t_0(a_1, z_6^i) \\
    \text{Otherwise } &t_0(a_1, z_5^i) = t_0(a_1, y_3^i) = t_0(a_1, b_0) = t_0(a_1, y_4^i) = t_0(a_1, z_6^i)
  \end{align*}We have that \( t_0(\aI{a}_1, \aI{z}_5) \) differs from \( t_0(\aI{a}_1, \aI{z}_6) \) in exactly two coordinates, which do not form a V-block; hence these two elements are not \(\Theta\)-congruent. This proves the claim.

  All that remains is to show that the edge relation is recoverable, so suppose \( v \Edge w \), \(|\chi|(\aI{y}_1) = v \) and \(|\chi|(\aI{y}_2) = w \). Let \( \chi(\aI{y}_v^+) = v^+, \chi(\aI{y}_v^-) = v^-, \chi(\aI{y}_w^+) = w^+, \chi(\aI{y}_w^-) = w^- \), and define \[ \aI{b}_{vw}^+ = {b_1}_{|v^+,w^+,\infty} \oplus {b_0}_{|\mathrm{else}} \qquad \aI{b}_{vw}^- = {b_1}_{|v^-,w^-,\infty} \oplus {b_0}_{|\mathrm{else}}\]We have \( \aI{b}_{vw}^+ , \aI{b}_{vw}^- \in \mathcal{B}_2\), \( \aI{y}_v^+, \aI{y}_w^+ \ll \aI{b}_{vw}^+ \), and \( \aI{y}_v^-, \aI{y}_w^- \ll \aI{b}_{vw}^- \). Next define \[ \aI{z}_{vw}^+ = {b_1}_{|v^+,w^+} \oplus {b_0}_{|\mathrm{else}} \qquad \aI{z}_{vw}^- = {b_1}_{|v^-,w^-} \oplus {b_0}_{|\mathrm{else}} \]Then \( \aI{z}_{vw}^+ \propto \aI{b}_{vw}^+ \), \( \aI{z}_{vw}^- \propto \aI{b}_{vw}^- \), and \[ t_0(\aI{a}_1, \aI{z}_{vw}^+) = {m_1}_{|v^+,w^+} \oplus {m_0}_{|\mathrm{else}} \equiv_\Theta {m_1}_{|v^-,w^-} \oplus {m_0}_{|\mathrm{else}} = t_0(\aI{a}_1, \aI{z}_{vw}^-) \]We have shown that for \( \aI{y}_1, \aI{y}_2 \in \mathcal{B}_1 \), \begin{align*}
    |\chi|(\aI{y}_1) \Edge |\chi|(\aI{y}_2) \; \Rightarrow \; &\exists v_3, \ldots, v_{10} \; \bigwedge_{3 \leq j \leq 6} v_j \in \mathcal{B}_1 \; \& \; \bigwedge_{7 \leq j \leq 8} v_j \in \mathcal{B}_2 \; \& \\
    &|\chi|(v_3) = |\chi|(v_4) = |\chi|(\aI{y}_1) \neq |\chi|(\aI{y}_2) = |\chi|(v_5) = |\chi|(v_6) \; \& \\
    &\chi(v_3) \neq \chi(v_4) \; \& \: \chi(v_5) \neq \chi(v_6) \; \& \\
    &v_3, v_5 \ll v_7 \; \& \; v_4,v_6 \ll v_8 \; \& \\
    &v_9 \propto v_7 \; \& \; v_{10} \propto v_8 \; \& \; t_0(a_1, v_9) = t_0(a_1, v_{10})
  \end{align*}Call this formula \( \mathrm{EDGE}(\aI{y}_1, \aI{y}_2) \) (again all variables \(v_3\) through \(v_{10} \) are secretly \(\ell\)-tuples).

  \begin{claim}
    The converse holds too; that is, the formula \( \mathrm{EDGE}(v_1,v_2) \) recovers the edge relation of \( \mathbb{G} \) on \( \mathcal{B}_1 / |\chi| \). 
  \end{claim}
	
  The proof is similar to the last claim's. Assume \( |\chi|(\aI{y}_1) \neq |\chi|(\aI{y}_2) \) and \( |\chi|(\aI{y}_1) \not \Edge |\chi|(\aI{y}_2) \). Let \( \aI{y}_3, \ldots, \aI{y}_8, \aI{z}_9, \aI{z}_{10} \) be as in the statement. Then since \( \aI{z}_9 \propto \aI{y}_7 \) and \( \aI{z}_{10} \propto \aI{y}_8 \), for all \( i \neq \infty \) we have \[ t_0(a_1, y_7^i) = t_0(a_1, z_9^i) \quad t_0(a_1, y_8^i) = t_0(a_1, z_{10}^i) \]By assumption, \( \Eset{\alg{S}}{b_0}{z_9^\infty} = A = \Eset{\alg{S}}{b_0}{z_{10}^\infty} \), so in particular \[ t_0(a_1, z_9^\infty ) = t_0(a_1, b_0) = t_0(a_1, z_{10}^\infty) \]Now for \( i \in V^\pm \) \begin{align*}
    \text{If } i \in \{ \chi(\aI{y}_3), \chi(\aI{y}_5) \} \\ \qquad \text{ then } & t_0(a_1, z_9^i) = t_0(a_1, y_7^i) \neq t_0(a_1, b_0) = t_0(a_1, y_8^i) = t_0(a_1, z_{10}^i) \\
    \text{If } i \in \{ \chi(\aI{y}_4), \chi(\aI{y}_6) \} \\ \qquad \text{ then } & t_0(a_1, z_9^i) = t_0(a_1, y_7^i) = t_0(a_1, b_0) \neq t_0(a_1, y_8^i) = t_0(a_1, z_{10}^i) \\
    \text{Otherwise } \\ & t_0(a_1, z_9^i) = t_0(a_1, y_7^i) = t_0(a_1, b_0) = t_0(a_1, y_8^i) = t_0(a_1, z_{10}^i) \\
  \end{align*}Hence \( t_0(\aI{a}_1, \aI{z}_9) \) differs from \( t_0(\aI{a}_1, \aI{z}_{10}) \) on a set of precisely four coordinates \( \{ v^+, v^-, w^+, w^- \} \) where \( v \not \Edge w \). It follows that \[ t_0(\aI{a}_1, \aI{z}_9) \not \equiv_\Theta t_0(\aI{a}_1, \aI{z}_{10}) \]which proves the Claim and the Lemma.
\end{proof}

\begin{lemma}\label{lemma:sigma abelian}
  The strongly solvable radical of every finite algebra lying in a finitely decidable variety is abelian. 
\end{lemma}

\begin{proof}
  Let \( \alg{S} \) be a counterexample of minimum possible cardinality, with strongly solvable radical \( \sigma \). We aim for a contradiction.

  \begin{claim}
    \( \alg{S} \) is subdirectly irreducible.
  \end{claim}
	
  To see this, let \begin{align*}
    t(a_1, \vec{b}_1) &= t(a_1, \vec{b}_2) \\ &\text{but} \\ t(a_2, \vec{b}_1) &\neq t(a_2, \vec{b}_2)
  \end{align*}witness \( \neg \TC{\sigma}{\sigma}{\bot_S} \), and let \( \alpha \) be maximal for separating \( t(a_2, \vec{b}_1) \) from \( t(a_2, \vec{b}_2) \). Then \( \alpha \) is meet-irreducible and \( \alpha \lor \sigma \) is strongly solvable over \( \alpha \). The same failure of the term condition shows that the strongly solvable radical of \( \alg{S} / \alpha \) is not abelian, which is incompatible with \( \bot_S < \alpha \). This proves the claim.

  Let \( \mu \) denote the monolith of \( \alg{S} \). Again by minimality, we also have that \( \TC{\sigma}{\sigma}{\mu} \). Of course, since \( \sigma \) is nontrivial, the monolith has unary type. By Lemma \ref{lemma:zeta solvable}, the centralizer of \(\mu\) is a strongly solvable congruence. We have that \(\alg{S}\) satisfies all the hypotheses of Lemma \ref{lemma:sigma centralizes minimal sets}, but by assumption, \( \HSP{\alg{S}} \) is not finitely undecidable; hence we must have that for all \( (\bot_S,\mu) \)-minimal sets \(U\), \( \TC{\sigma}{\mu_{|U}}{\bot} \).

  Now by Lemma \ref{lemma:twins act trivially}, we have that for any \( (\bot,\mu) \)-minimal set \(U\), the action of \( \TwinGrp{\alg{S}}{U} \) inside any trace \( N \subset U \) is trivial.

  \begin{claim}\label{claim:[mu,sigma] = bot}
    \( \TC{\mu}{\sigma}{\bot} \); equivalently, \( [\mu, \sigma] = \bot \).
  \end{claim}

  Suppose otherwise. Choose a witnessing package \begin{align*}
    t(a_1, \vec{b}_1) &= t(a_1, \vec{b}_2) \\ &\text{but} \\
    t(a_2, \vec{b}_1) &\neq t(a_2, \vec{b}_2)
  \end{align*}such that \(a_1, a_2\) belong to some trace \(N\) inside a \( (\bot, \mu) \)-minimal set \(U\) and the polynomial \( t(v_0, \ldots, v_k ) \) respects \(U\). Then it is not possible for either of the functions \[ f_i(v_0) = t(v_0, \vec{b}_i) \](\(i = 1,2\)) to collapse traces to points; hence these two funtions are twin elements of \( \PolGrp{\alg{S}}{U} \).
	
  But then the first line (equality) says that \( f_2^{-1} \circ f_1 (a_1) = a_1 \), implying that \( f_2^{-1} \circ f_1 (N) = N \); but the second line yields \( f_2^{-1} \circ f_1 (a_2) \neq a_2 \). This contradiction proves the claim.

  By Theorem 4.5 of \cite{Kearnes93}, Claim \ref{claim:[mu,sigma] = bot} implies that \(\mu\) is \(\sigma\)-coherent. We have already shown that the hypothesis of the coherence property, \( \bigwith_N \TC{\sigma}{\mu_{|N}}{\bot} \), holds; hence we have both \( \TC{\mu}{\sigma}{\bot} \) and \( \TC{\sigma}{\mu}{\bot} \).

  This shows that \( \alg{S} \) satisfies all the hypotheses of Lemma \ref{lemma:C(sigma,sigma;bot)}. Since our assumption was that \( \HSP{\alg{S}} \) is not finitely undecidable, we must have \( \TC{\sigma}{\sigma}{\bot} \). But this contradicts our choice of \( \alg{S} \) as a counterexample.
\end{proof}

\begin{lemma}\label{lemma:nonSA implies nonA}
  If \(\alg{F}\) is a finite algebra with a strongly solvable congruence which is abelian but not strongly abelian, then \(HS(\alg{F}^2)\) contains an algebra with a strongly solvable congruence which is not abelian.
\end{lemma}

\begin{proof}
  Let \(\sigma \in \Con{\alg{F}}\) be strongly solvable and abelian, but not strongly abelian. If \(\sigma\) is not abelian over some congruence beneath it, we are done; so without loss of generality \(\sigma\) is strongly abelian over every nontrivial congruence \(\bot < \alpha \leq \sigma\). (Else pass from \(\alg{F}\) to its quotient by a congruence maximal for \(\sigma\) not being strongly abelian over it.) We have that \[\sigma \times \sigma = \eta_1^{-1}(\sigma) \land \eta_2^{-1}(\sigma) \]is a strongly solvable congruence of \(\alg{F}^2\).
	
  Let \begin{align*}c_1 = t(a_1,\vec{b}_1) &\neq t(a_1,\vec{b}_2) = c_3 \\ c_2 = t(a_2,\vec{b}_1) &\neq t(a_2, \vec{b}_2) = c_1	
  \end{align*}witness the failure of strong abelian-ness of \(\sigma\) over \(\bot_F\). Since \( \sigma \) is strongly abelian over every nontrivial \( \alpha \leq \sigma \), it follows that \( c_1 \equiv_\alpha c_2 \equiv_\alpha c_3 \) for all such \( \alpha \); in particular, there is only one congruence atom \( \mu = \Cg{}{c_1}{c_2} = \Cg{}{c_1}{c_3} \) below \( \sigma \).

  \stepcounter{claim}
  Since \( \TC{\sigma}{\sigma}{\bot} \), for any polynomial \( p(x) \in \Pol{1}{\alg{F}} \) we have \begin{align*}
      p(c_1) = p(t(a_1, \vec{b}_1)) &= p(t(a_2, \vec{b}_1)) = p(c_2) \\ &\Updownarrow \tag{\theclaim}\label{eq:polys acting on ci} \\
      p(c_3) = p(t(a_1, \vec{b}_2)) &= p(t(a_2, \vec{b}_2)) = p(c_1)
    \end{align*}

  \stepcounter{claim}
  Our proof will proceed somewhat differently depending on whether \(a_1, a_2 \) could be chosen \( \mu \)-equivalent. If this is not possible, then for all polynomials \(s\) and all \( m_1 \equiv_\mu m_2 \) and \(\vec{u}_1 \equiv_\sigma \vec{u}_2\), \[\tag{\theclaim} \label{eq:a1 not mu to a2} s(m_1, \vec{u}_1) = s(m_2, \vec{u}_2) \; \Rightarrow \; s(m_1, \vec{u}_1) = s(m_1, \vec{u}_2) = s(m_2, \vec{u}_1) \]

  In both cases, let \( \alg{C} \leq \alg{F}^2 \) be the subalgebra generated by the diagonal together with \( \binom{a_1}{a_2} \). Then as subalgebras, \( \alg{C} \leq \sigma \), and if \( a_1 \equiv_\mu a_2 \) then \( \alg{C} \leq \mu \). Let \( \beta \in \Con{\alg{C}} \) be generated by identifying \( \binom{c_1}{c_2} \equiv_\beta \binom{c_3}{c_1} \). We will show that \( \sigma \times \sigma \) is not abelian over \( \beta \).

  \begin{claim}
    \( \binom{c_1}{c_1} \) is isolated mod \( \beta \); that is, there do not exist \(f \in \Pol{1}{\alg{F}}\) and \( \binom{e_{1i}}{e_{2i}} \in \alg{C} \) such that\[ \binom{c_1}{c_1} = \binom{f(c_1, \vec{e}_1)}{f(c_2, \vec{e}_2)} \neq \binom{f(c_3, \vec{e}_1)}{f(c_1, \vec{e}_2)} = \binom{d_1}{d_2} \]
  \end{claim}

  Suppose first that \( a_1 \) could not be chosen \( \mu \)-congruent to \( a_2 \). By equation \eqref{eq:a1 not mu to a2}, \( c_1 = f(c_2, \vec{e}_1) = f(c_1, \vec{e}_2) = d_2 \); it follows by equation \eqref{eq:polys acting on ci} that \( c_1 = f(c_1, \vec{e_1}) = f(c_3, \vec{e}_1) = d_1 \). This contradiction proves the first case of the claim.

  In the other case, assume that \( a_1 \equiv_\mu a_2 \), so that \( \alg{C} \) is a subalgebra of \(\mu\), which is a strongly abelian congruence. The equality \( f(c_1, \vec{e}_1) = f(c_2, \vec{e}_2) \) implies that \[ c_1 = f(c_2, \vec{e}_2) = f(c_2, \vec{e}_1) = f(c_1, \vec{e}_2) = d_2 \]Equation \eqref{eq:polys acting on ci} implies that \( f(c_3,\vec{e}_1) = c_1\) too.

  With the previous claim in place, the following failure of the term condition \begin{align*}
    \binom{c_1}{c_2} = t\left( \binom{a_1}{a_2}, \binom{\vec{b_1}}{\vec{b_1}} \right) &\equiv_\beta t\left( \binom{a_1}{a_2}, \binom{\vec{b_2}}{\vec{b_2}} \right) = \binom{c_3}{c_1} \\
    \binom{c_2}{c_2} = t\left( \binom{a_2}{a_2}, \binom{\vec{b_1}}{\vec{b_1}} \right) &\not \equiv_\beta t\left( \binom{a_2}{a_2}, \binom{\vec{b_2}}{\vec{b_2}} \right) = \binom{c_1}{c_1}\end{align*}shows that \(\sigma \times \sigma \) is not abelian over \(\beta\).
\end{proof}

We are ready to finish this section's main result:

\begin{proof}[Proof of Theorem \ref{MainThm:sigma strongly abelian}]
  By Lemma \ref{lemma:nonSA implies nonA}, if \( \alg{A} \) is any finite algebra whose strongly solvable radical is not strongly abelian, then \( \HSP{\alg{A}} \) contains a finite algebra whose strongly solvable radical is nonabelian. By Lemma \ref{lemma:sigma abelian}, such an algebra cannot lie in any finitely decidable variety.
\end{proof}

\section[Bounding SIs in \( \var{V} \)]{The finite residual bound on a finitely decidable variety}\label{sec:ResidualBound}
\setcounter{thm}{0}

We now are ready for the proof of Theorem \ref{MainThm:residual bound}. For the remainder of this section, fix a finitely generated, finitely decidable variety \(\var{V} \), say \( \var{V} = \HSP{\var{K}} \), where \( \var{K} \) is a finite set of finite algebras.

\begin{lemma}\label{lemma:type 3 bound}
  \( \var{V} \) contains only finitely many subdirectly irreducible finite algebras whose monolith is of boolean type.
\end{lemma}

\begin{proof}
  We will show that in fact every finite subdirectly irreducible \[\alg{S} \in \HSP{\var{K}}\] with boolean-type monolith already belongs to \(\HS{\var{K}}\).
	
  So let \( \alg{S} \) be a quotient of \[\alg{B} \leq \prod_{i = 1}^p A_i\]where each \(\alg{A}_i \in \var{K}\) and \(p\) is the smallest number of factors for which such a representation exists; say \(\alg{S} \cong \alg{B}/\pi\), where \(\pi\) is meet-irreducible, with upper cover \(\mu\) such that \( \typ{\pi} {\mu} = 3\). The minimality of \(p\) implies that each \(\hat{\eta}_{i} = \bigwedge_{j \neq i} \eta_j\) has no congruence \(\theta\) above it such that \(\alg{B}/\theta \cong \alg{S}\); in particular, for each \(i\), \(\hat{\eta}_{i} \lor \pi \geq \mu\).

  Choose some \((\pi,\mu)\)-minimal set \(U = e(B)\). Then \(U\) has empty tail and only one trace, so \(U = \{\aI{x},\aI{y}\}\). Let \(\beta = \Cg{}{\aI{x}}{\aI{y}}\), and observe that \(\mu = \pi \lor \beta\).	

  \begin{claim}
    \(\Con{\alg{B}} = I[\bot,\pi] \sqcup I[\beta,\top]\).
  \end{claim}
	
  The disjointness is obvious. Let \(\theta \not \leq \pi\). Then \(\theta \lor \pi \geq \mu\), and in particular identifies \(\aI{x}\) and \(\aI{y}\). String a chain of elements between them: \[\aI{x} \equiv_\theta \aI{z}_1 \equiv_\pi \aI{z}_2 \equiv_\theta \cdots \equiv_\pi \aI{z}_n \equiv_\theta \aI{y}\]and hit this chain with \(e\): \[\aI{x} = e(\aI{x}) \equiv_\theta e(\aI{z}_1) \equiv_\pi e(\aI{z}_2) \equiv_\theta \cdots \equiv_\pi e(\aI{z}_n) \equiv_\theta e(\aI{y}) = \aI{y}\]The resulting chain is in \(U\), so the \(\pi\)-links are trivial, implying that \(\aI{x} \equiv_\theta \aI{y}\), as claimed.
	
  We have already seen that \(\hat{\eta}_i \not \leq \pi\) for any \(1 \leq i \leq p\); by the claim, each \(\hat{\eta}_i\) identifies \(\aI{x} \) and  \( \aI{y}\). But now observe that if \(p\) were to be greater than \(1\), we would have \[\langle \aI{x},\aI{y} \rangle \in \hat{\eta}_1 \cap \hat{\eta}_2 = \bot\]which would be absurd. Hence \(p = 1\) and the theorem follows. 
\end{proof}

\begin{lemma}\label{lemma:type 2 bound}
  \( \var{V} \) contains only finitely many subdirectly irreducible finite algebras whose monolith is of affine type.
\end{lemma}

The proof adapts from, but corrects an error in, \cite{Snow_McKenzie} Section 12. 

\begin{proof}
  Let \( \alg{S} \in \HSP{\var{K}} \) be subdirectly irreducible with affine monolith; say \( \alg{S} = \alg{B}/\pi\), where \[ \alg{B} \leq_s \prod_{i=1}^p \alg{A}_i \qquad (\alg{A}_i \in \var{K})\]Without loss of generality \( \var{K} = \HS{\var{K}} \), and the representation is minimal in the sense that \( \alg{S} \) is not representable in this way by fewer than \(p\) factors from \( \var{K}\), and moreover if \( \beta_i \in \Con{\alg{A}_i} \) and \(\alg{S} \) is a quotient of a subalgebra of \( \prod_i \alg{A}_i/\beta_i \) then all \( \beta_i \) are trivial.

  \begin{claim}
    Let \( \sigma_i \) denote the strongly solvable radical of \(\alg{A}_i\), and \( \sigma_1 \times \cdots \times \sigma_p = \sigma \in \Con{\alg{B}}\). Then \(\sigma \leq \pi\).
  \end{claim}

  Suppose this were false. Let \[ \bot_B \leq \alpha^- \prectype{1} \alpha^+ \leq \sigma \]such that \( \alpha^- \leq \pi \) but \( \pi \leq \beta^- \prectype{2,3} \beta^+ = \alpha^+ \lor \pi \). Then the covers \( \alpha^- \prectype{1} \alpha^+ \) and \( \beta^- \prectype{2,3} \beta^+ \) are projective, which is absurd (cf. Theorem \ref{thm:SSsim is congruence}).

  Our minimality assumption implies now that each \( \alg{A}_i\) in the representation of \( \alg{B} \) has trivial strongly solvable radical. By Lemma \ref{lemma:subpowers omit type 1}, \( \alg{B} \) has Day polynomials; hence the term condition on congruences of \( \alg{B} \) is symmetric in the first two variables. 

It follows by Theorem 10.1 of \cite{Freese_McKenzie} that \( \alg{S}/\zeta \in \HS{\var{K}}\), where \( \zeta \) denotes the centralizer of the monolith \(\mu\); in particular, \[ \card{\alg{S}/\zeta} \leq \max \{ \card{A} \colon \alg{A} \in \var{K} \} \]
  
  We will be done if we can show that there is also a bound on the number of elements of each \(\zeta\)-block. From now on we will forget about \(\alg{B}\) and work only in \(\alg{S}\). Let \(\{C_i = r_i/\zeta : 1 \leq i \leq \ell \}\) be an injective enumeration (with fixed representatives) of the \(\zeta\)-classes, \(C\) any fixed one of them, and \(U\) a \((\bot_S,\mu)\)-minimal set containing a monolith pair \(\{0,a\}\).

  As before, we have a Malcev polynomial \(m(v_1,v_2,v_3)\) on \(U\); furthermore, if \(Q \subseteq U\) denotes the \(\zeta\)-class of \(0\) in \(U\), then \(m\) respects \(Q\). Since the tail of \(U\) is empty, \(\alg{S}_{|U}\) is then an abelian Malcev algebra. By a standard argument, the operation \( m(x,y,z) = x - y + z \) defines an abelian group operation on \(Q\) under which \(0\) is the identity element.

  \begin{claim}
    The set of polynomial functions \[R = \{f(v) \in \Pol{1}{\alg{S}_{|Q}} : f(0) = 0 \} \]is a ring of endomorphisms of \(\alg{Q}\) (under pointwise addition and function composition), and the size of \(R\) is bounded independent of \(\alg{S}\).
  \end{claim}
	
  The only nontrivial piece of the first part is that each such \(f\) respects addition: \begin{align*}
    f(x) = f(x - y + y) + f(0) &= f(y - y + 0) + f(x) = f(x) \\ &\Downarrow \\
    f(x+y) = f(x - 0 + y) + f(0) &= f(y - 0 + 0) + f(x) = f(y) + f(x)
  \end{align*}The second comes from the fact that each \(f \in R\) is given by an \(\ell + 1\)-ary term operation in a uniform way: if \(f(x) = t(x,\vec{s})\) then \begin{align*}
    0 = t(0,\vec{s}) - t(0,\vec{s}) &= t(0,\vec{r}) - t(0,\vec{r}) \\ &\Downarrow \\
    t(x,\vec{s}) = t(x,\vec{s}) - t(0,\vec{s}) &= t(x,\vec{r}) - t(0,\vec{r})
  \end{align*}where \(\vec{r}\) denotes the chosen representatives of the \(\zeta\)-classes. Hence \(|R| \leq |\alg{F}_\var{V} (1 + \ell) |\).
  
  Now: for any \(s_1 \neq s_2 \in S\), there exists a polynomial \(f(v_0) = t(v_0,\vec{s})\) so that \(t(s_1,\vec{s}) = 0\) and \(t(s_2,\vec{s}) = a\). In particular, if \(s_1 = 0, s_2 \in Q\) then we may take \(f \in R\). 

  What this shows is that \( \alg{Q} \) is subdirectly irreducible as an \( R \)-module. By Theorem 1 of \cite{Kearnes91}, \( \card{Q} \leq \card{R} \).

  Now we are almost done: we have already noted that for each \(c,d \in C\) there exists a term \(t(v_0,\ldots,v_{|S|})\) with \(t(c,\vec{s}) = 0\), \(t(d,\vec{s}) = a\). One has 
  \begin{align*}
    et(d,\vec{s}) - et(d,\vec{s}) &= et(d,\vec{r}) - et(d,\vec{r}) \\ &\Downarrow \\
    a = et(d,\vec{s}) - et(c,\vec{s}) &= et(d,\vec{r}) - et(c,\vec{r})
  \end{align*}where all these values must lie in \(Q\). Hence the map 
  \begin{align*}
    C &\rightarrow \fourIdx{\alg{F}_\var{V}(1+\ell)}{}{}{}{Q} \\
    x &\mapsto \langle et(x,\vec{r}) \colon t \in \alg{F}_\var{V}(1+\ell) \rangle
  \end{align*}is injective.

  We have shown \[ \card{C} \leq \card{Q}^{\card{\alg{F}_{\var{V}}(1+\ell)}} \leq \card{R}^{\card{\alg{F}_{\var{V}}(1+\ell)}} \leq \card{\alg{F}_{\var{V}}(1+\ell)}^{\card{\alg{F}_{\var{V}}(1+\ell)}} \]which, combined with the fact that \[ \card{S} \leq \card{C} \cdot \max_{\alg{A} \in \mathcal{K}}(\card{A}) \]completes the proof.
\end{proof}

We will need the following technical lemma limiting the number of variables which can be independent (modulo a strongly abelian congruence) in a polynomial operation.

\begin{lemma}\label{lemma:few independent variables}
  Let \( \alg{A} \) be a finite algebra in a locally finite variety \( \var{V} \), and \( \beta \) a strongly abelian congruence on \( \alg{A} \), and \( t(v_0, \vec{v}_1, \ldots, \vec{v}_\ell) \) be any polynomial operation of \( \alg{A} \). Let \(M = \log\card{\alg{F}_{\var{V}}(\ell + 2)}\). Then there exist subsets \( \breve{v}_i \subset \vec{v}_i \) of size no more than \( M \), such that for any \( \beta \)-blocks \( B_1, \ldots, B_\ell \) the mapping \stepcounter{claim}\begin{align}
    A \times \vec{B}_1 \times \cdots \times \vec{B}_\ell &\rightarrow A \notag \\
    \langle a, \vec{b}_1, \ldots, \vec{b}_\ell \rangle &\mapsto t(a, \vec{b}_1, \ldots, \vec{b}_\ell) \label{align:fewvars}
  \end{align}depends only on the variables \(v_0 \) and \( \breve{v}_i \).
\end{lemma}

\begin{proof}
  For simplicity, we show the case \(\ell = 2\). Let \(t(v_0, v_1^1, \ldots, v_1^{k_1}, v_2^1, \ldots, v_2^{k_2}) \) be our term, and let \( 2^{k_1} > \card{\alg{F}_{\var{V}}(4)} \).
	
  For \( S \subset \{ 1, \ldots, k_1 \} \) let \( p_S(v_0, x, y, v_2) \) be the substitution instance of \(t\) obtained by identifying all \(v_2^i\) to the single variable \(v_2\), and substituting \(x\) for \(v_1^i\) if \(i \in S\) and \(y\) if not. Then by Pigeonhole, there exist \( S \neq S' \) so that \( \var{V} \models p_S = p_{S'}\). Say \(k_1 \in S\) but not \(S'\); we claim that no mapping as in (\ref{align:fewvars}) can depend on \( v_1^{k_1} \).

  To see this, let \(a \in A\), \(b,c \in B_1 \), and \(d \in B_2\). Let \( q_S(v_0, x,y,v_1^{k_1}, v_2) \) be like \( p_S \), except that \(v_1^{k_1} \) is left unsubstituted, and likewise for \(q_{S'}\). Then \[q_S(a,b,c,b,d) = q_{S'}(a,b,c,c,d) \]But now since \( \beta\) is strongly abelian, if \( \vec{x} \equiv_\beta b \) and \( \vec{y} \equiv_\beta d\), the strong term condition gives that \[t(a,\vec{x},b,\vec{y}) = t(a,\vec{x},c,\vec{y}) \]so \(t\) is insensitive to changes modulo \(\beta\) in the \(v_1^{k_1}\) coordinate. Similarly, if \( 2^{k_2} > \card{\alg{F}_{\var{V}}(4)} \) then \(t\) is insensitive to changes mod \( \beta\) in some coordinate \(v_2^i\). The general result now follows by a downward induction.
\end{proof}

\begin{lemma}\label{lemma:type 1 bound}
  \( \var{V} \) contains only finitely many subdirectly irreducible finite algebras whose monolith is of unary type.
\end{lemma}

\begin{proof}
  Let \( \alg{S} \in \var{V} \) be subdirectly irreducible with unary-type monolith \[ \mu = \Cg{\alg{S}}{c}{d}\]We already know that \( \typset{\alg{S}} \subset \{1,3\} \). By Theorem \ref{MainThm:sigma strongly abelian}, the strongly solvable radical \(\sigma\) is a strongly abelian congruence. Either \( \sigma = \top_S \) or, by Corollary \ref{cor:sigma meet-irreducible}, \(\sigma\) is meet-irreducible with upper cover of boolean type. In either case, \[ \ell := \card{\alg{S}/\sigma} \leq M_{\mathrm{bool}} \](where \( M_{\mathrm{bool}} \) denotes the maximum cardinality of a finite SI in \( \var{V} \) with boolean-type monolith). Fix some enumeration \( \langle \vec{s}_1, \ldots, \vec{s}_\ell \rangle \) of \(S\) with each \( \sigma \)-block \(B_i\) enumerated together. We must now put a uniform bound on the size of \( \sigma \)-blocks.
	
  Let \(B\) be any \(\sigma\)-block. Since any unequal pair of elements generates a congruence above \(\mu\), we have that for any \( b \neq b' \in B \), there exists a unary polynomial \( p(v_0) = t(v_0, \vec{s}_1, \ldots, \vec{s}_\ell) \) such that \( p(b) = c \; \mathrm{iff} \; p(b') \neq c \). By Lemma \ref{lemma:few independent variables}, these terms depend (up to changes mod \(\sigma\)) on \(v_0\) and subsets \( \breve{s}_i \), \( 1 \leq i \leq \ell \), each of size no more than \( M := \log(\alg{F}_{\var{V}}(\ell + 2)) \). Let \( P = B_1^M \times \cdots \times B_\ell^M \).
	
  For \(b \in B\), we define a subset \( G(b) \subset \alg{F}_{\var{V}}(1 + \ell M)\) to consist of those terms \(t(x,\vec{y}) \) such that for some \( \vec{p} \in P \), \( t(b,\vec{p}) = c \).
	
  \begin{claim}
    The mapping \( b \mapsto G(b) \) is injective.
  \end{claim}We will be done once we have established the claim, since then \[ \card{B} \leq 2^{\card{\alg{F}_{\var{V}}(1 + \ell M)}} \]which is uniformly bounded in \( \var{V} \).
	
  To prove the claim, let \(b_1 \neq b_2\), and assume towards a contradiction that \( G(b_1) = G(b_2)\). At least one, and hence both, must be nonempty. Choose a term \(t\) and a \( \vec{p}_1 \in \Sigma \) so that \( c = t(b_1, \vec{p}_1) \neq t(b_2, \vec{p}_1 ) \). Then \( t \in G(b_1) = G(b_2) \), so we can choose \( \vec{p}_2 \in \Sigma \) so that \( t(b_2, \vec{p}_2) = c \). Hence we have a failure \[ \begin{matrix} c = t(b_1, \vec{p}_1) & t(b_1, \vec{p}_2) \\c \neq t(b_2, \vec{p}_1) & t(b_2, \vec{p}_2) = c \end{matrix} \]of the strong term condition, since the entries are equal along the diagonal but not along the rows and columns. This contradicts the strong abelianness of \(\sigma\).
\end{proof}

\begin{proof}[Proof of Theorem \ref{MainThm:residual bound}]
  Since \(\var{V} \) is locally finite, it is enough to prove that \( \var{V} \) contains only finitely many finite subdirectly irreducible algebras. (It is a well-known result, originally due to Quackenbush, that an infinite SI algebra in a locally finite variety has arbitrarily large finite SI subalgebras generated by a monolith pair together with other elements.) Since \( \var{V} \) is finitely decidable, it omits the semilattice and lattice types altogether; and Lemmas \ref{lemma:type 3 bound}, \ref{lemma:type 2 bound}, and \ref{lemma:type 1 bound} combine to show that there are only finitely many SIs in \( \var{V} \) with monoliths of the boolean, affine, or unary types.
\end{proof}

%
%

\end{document}